\documentclass[11pt,a4paper]{article}
\usepackage{amsthm,amsmath}
\usepackage{mathrsfs}
\usepackage{amssymb}
\usepackage{latexsym}
\usepackage{yfonts}

\usepackage{graphicx}
\usepackage{color}
\def\E{{\mathbb E}}
\def\P{{\mathbb P}}

\newcommand{\RR}[1]{\mathbb{#1}}
\newtheorem{te}{Theorem}[section]

\theoremstyle{definition}

\theoremstyle{os}
\newtheorem{os}[te]{Remark}

\theoremstyle{prop}
\newtheorem{prop}[te]{Proposition}

\theoremstyle{lem}
\newtheorem{lem}[te]{Lemma}

\theoremstyle{coro}
\newtheorem{coro}[te]{Corollary}

\numberwithin{equation}{section}

\title{Time dependent random fields on spherical non-homogeneous surfaces}

\author{Mirko D'Ovidio\footnote{Dipartimento di Scienze di Base e Applicate per l'Ingegneria, Sapienza Universit\`{a} di Roma, Via A. Scarpa 16, 00161 Roma, Italy. E-mail:mirko.dovidio@sbai.uniroma1.it}  and Erkan Nane\footnote{221 Parker Hall, Department of Mathematics and Statistics,
Auburn University, Auburn, Al 36849.
E-mail: nane@auburn.edu. http://www.auburn.edu/$\sim$ezn0001}}

\begin{document}

\maketitle

\begin{abstract}
We introduce a class of isotropic time dependent random fields on the non-homogeneous sphere represented by a time-changed spherical Brownian motion of order $\nu \in (0,1]$ with which some anisotrophies   can be captured in Cosmology. This process is a time-changed rotational diffusion (TRD) or the stochastic solution to the equation involving the spherical Laplace operator and a time-fractional derivative of order $\nu$.  TRD is a diffusion on the non-homogeneous sphere and therefore, the spherical coordinates given by  TRD represent the coordinates of a non-homogeneous sphere by means of which an isotopic random field is indexed. The time dependent random fields we present in this work is therefore realized through composition and can be viewed as isotropic random field on randomly varying sphere.
\end{abstract}

{\bf keyword:}random field on the sphere, time-changed rotational Brownian motion, stable subordinator, fractional diffusion, Wigner coefficient, Clebsch-Gordan coefficient, CMB radiation.

\tableofcontents

\section{Introduction and statement of main results}
In recent years a growing literature has been devoted to the study of the random fields on the sphere and their statistical analysis. Many researchers have focused on the construction and characterization of random field indexed by compact manifolds  such as the sphere $\mathbb{S}^2_r = \{\mathbf{x} \in \mathbb{R}^3:\, |\mathbf{x}|=r\}$, see for example \cite{balMar07, marin06, marin08, marpec08}. In such papers the sphere represents a homogeneous surface in which the random field is observed. The interest in studying random fields on the sphere is especially represented by the analysis of the Cosmic Microwave Background (CMB) radiation which is currently at the core of physical and cosmological research, see for instance \cite{Dodelson, kolturner}. CMB radiation is thermal radiation filling the observable universe almost uniformly \cite{PenWil65} and is well explained as radiation associated with an early stage in the development of the universe. From a mathematical viewpoint the CMB radiation can be interpreted as a realization of an isotropic, mean-square continuous spherical random field for which a spectral representation given by means of spherical harmonics holds. Due to the Einstein cosmological principle (on sufficiently large scales,  the universe looks identical everywhere in space (homogeneity) and appears the same in every direction (isotropy)) the CMB radiation is an isotropic image of the early universe \cite{Dodelson}. Nevertheless, such a nature of the CMB radiation can be affected by anisotropies as those due to the gravitational lensing for instance. Recently, a growing attention has been drawn to high-frequency or equivalently high-resolution asymptotics for statistics based upon functionals of isotropic random fields (see for instance \cite{BKMP09, marin06}).

Beside the interest on random fields, particular attention has been also paid, in last years, by yet other researchers in studying fractional diffusion equations. These equations are related with anomalous diffusions or diffusions in non-homogeneous media, with random fractal structures for instance \cite{meerschaert-nane-xiao}. Starting from the works  \cite{Koc89, Nig86, Wyss86} much effort has been made in order to introduce a rigorous mathematical approach (see for example \cite{NANERW} for a short survey on this results). The solutions to fractional diffusion equations are strictly related with stable densities. Indeed, the stochastic solutions we are dealing with can be realized through time-change that are inverse stable subordinators and therefore we obtain time-changed  processes.   A couple of recent works in this field are \cite{BMN09ann, OB09}.

Let ${D}^\nu_t g(t)$ with $\nu \in (0, 1]$, be the Dzhrbashyan-Caputo fractional derivative of $g(t)$ of order $\nu$  defined by  its Laplace transform $\int_0^\infty e^{-st}{D}^\nu_t g(t)dt=s^\nu \tilde g(s)-s^{\nu-1}g(0)$ where $\tilde g(s)=\int_0^\infty e^{-st} g(t)dt$ is the Laplace transform of $g(t)$, and ${D}^\nu_t g(t)$ becomes the ordinary  first derivative $\partial g(t) / \partial t $ for $\nu=1$.

 Let $\mathfrak{H}^\nu_t$ be a stable subordinator of index $\nu\in (0,1)$ with Laplace transform
 \begin{equation}
 \mathbb{E} \exp ( -s  \mathfrak{H}^{\nu}_t)  = \exp (- t \, s^\nu). \label{lapH}
\end{equation}  We define by
 \begin{equation}\label{inverse-stable}
 \mathfrak{L}^{\nu}_t=\inf\{\tau>0:\mathfrak{H}^{\nu}_\tau>t \}
 \end{equation}
    the inverse of the stable subordinator $\mathfrak{H}^\nu_t$ of order $\nu \in (0,1)$.
   $\mathfrak{L}^\nu_t$ has non-negative, non-stationary and non-independent increments (see \cite{MSheff04}).
  Let
 \begin{equation}\label{mittag-leffler-function}
  E_{\nu}(z)=\sum_{n=1}^\infty \frac{z^n}{\Gamma(1+\nu)}
  \end{equation} be  the Mittag-Leffler function. By equation (3.16) in \cite{BMN09ann} the Laplace transform of $\mathfrak{L}^{\nu}_t$ is given by
\begin{equation}
\mathbb{E} \exp (- \lambda \mathfrak{L}^{\nu}_t)  = E_{\nu}(-\lambda t^\nu).  \label{lapL}
\end{equation}

 Here and in the sequel  we use the fact that $x\in \mathbb{S}_{r}^2$ can be represented as $$x=(r\sin\vartheta \cos \varphi, r\sin\vartheta \sin \varphi, r\cos \vartheta ).$$
 The spherical Laplacian is defined by
\begin{align*}
\triangle_{\mathbb{S}_{r}^2} = & \frac{1}{r^2} \left[ \frac{1}{\sin^2 \vartheta} \frac{\partial^2}{\partial \varphi^2} + \frac{1}{\sin \vartheta} \frac{\partial}{\partial \vartheta} \left( \sin \vartheta \frac{\partial}{\partial \vartheta} \right) \right], \ \ \vartheta \in [0,\pi],\; \varphi \in [0, 2\pi].
\end{align*}
Sometimes $\triangle_{\mathbb{S}_{r}^2}$ is called Laplace operator on the sphere.
 For the sake of simplicity we will consider $r=1$, that is the sphere of radius one.

Since the density of   Brownian motion $\mathfrak{B}^x(t)$ on the sphere $\mathbb{S}^2_1$ started at $x\in \mathbb{S}^2_1$  solves equation
\begin{equation}\label{brownian-pde}
\frac{\partial \,u(x, t) }{\partial t}=\partial_t u(x,t)= \triangle_{\mathbb{S}^2_1}\, u(x, t), \quad x \in \mathbb{S}^2_1,\; t>0
\end{equation}
we say $\mathfrak{B}(t)$, $t>0$ is a stochastic solution to \eqref{brownian-pde}.

Let $\langle x, y \rangle=\cos d(x,y)$ be the usual inner product in $\mathbb{R}^3$. With the notation  $x=(\sin\vartheta_x \cos \varphi_x, \sin\vartheta_x \sin \varphi_x, \cos \vartheta_x )$ and $y=(\sin\vartheta_y \cos \varphi_y, \sin\vartheta_y \sin \varphi_y, \cos \vartheta_y )$ we have
$$
\cos d(x,y)=\cos \vartheta_x \cos \vartheta_y+\sin\vartheta_x \sin\vartheta_y\cos (\varphi_x-\varphi_x)=\langle x, y \rangle.
$$

 Let $P_l$, $l = 0,1,2,\cdots $ be the  Legendre functions associated to the eigenvalue problem $\Delta_{\mathbb{S}^2_1}u=-\mu_l u$.  Let $\{R_l,\,l\geq 0\}$ be the set of positive real numbers depending on the angular power spectrum (defined in section \ref{SecRBM}) of $u_0$ with $R_0=1$.
The space $H^{s}(\mathbb{S}^2_1)$ is a subset of $L^2(\mathbb{S}_{1}^2)$ and it is defined in equation \eqref{SobolevS}.

Our interest, in this work, is to study the stochastic solution to a time-fractional Cauchy problem on $\mathbb{S}^2_1$ and obtain a new random structure for the sphere by means of which the random field is indexed.

\begin{te}Let  $\nu \in (0,1)$ and $s>7/2$.
The unique strong solution to the fractional Cauchy problem
\begin{equation}\label{pdeFracSphere}
\begin{split}
\partial^{\nu}_{t} u_\nu(x,t; x_0, t_0) &= \triangle_{\mathbb{S}_{1}^2} u_{\nu}(x,t; x_0, t_0), \quad x, x_0 \in \mathbb{S}^2_{1}, \, 0 \leq t_0 < t \\
u_{\nu}(x, t_0; x_0, t_0) &= u_0(x - x_0), \quad u_0(\cdot - x_0) \in H^{s}(\mathbb{S}^2_1), \;
\end{split}
\end{equation}
 is given by
\begin{equation}
u_{\nu}(x, t; x_0, t_0) =\mathbb{E}u_0\left(\mathfrak{B}^{x}(\mathfrak{L}^\nu_{t-t_0})-x_0 \right)
\end{equation}
where $\mathfrak{B}^{x}_{\nu}(t)$, $t>t_0$ is the time-changed rotational Brownian motion on the sphere $\mathbb{S}^2_1$ started at $x$ at time $t_0$. The explicit solution is written as
\begin{equation}
u_{\nu}(x, t; x_0, t_0) = \sum_{l \geq 0} \frac{2l+1}{4\pi} R_l\, E_{\nu}\left( - \mu_l\, (t-t_0)^{\nu} \right)  P_l(\langle x, x_0 \rangle) \label{lawFracRot}
\end{equation}
\label{TheoremMain}
\end{te}

This theorem is an extension of the results in \cite{BMN09ann} and \cite{meerschaert-skorski} to the spherical Laplacian on  the sphere.

$\mathfrak{B}^{x_0}_\nu(t)=\mathfrak{B}^{x_0}(\mathfrak{L}^\nu_t)$ is called  time-changed rotational Brownian motion on $\mathbb{S}^2_1$. $\mathfrak{B}^{x_0}_\nu(t)$  is a stochastic solution to \eqref{pdeFracSphere} and  has non-independent, non-stationary and non-negative increments.  Furthermore, for $\nu \to 1$, we get that a.s. $\mathfrak{L}^\nu_t \to t$ which is the elementary subordinator (see \cite{Btoi96}).
 When $\nu=1$ the time-changed rotational Brownian motion becomes the rotational Brownian motion or Brownian motion on the sphere $\{\mathfrak{B}(t)$, $t>0$.

We next consider a real-valued  random field  on the sphere $\{T(x):\, x \in \mathbb{S}^2_1\}$. Let $SO(3)$ be  the group of rotations  in $\mathbb{R}^{3}$ that can be realized as the space of $3\times 3$ real matrices $A$ such that $A'A=I_3$ (where $I_3$  is the three-dimensional identity matrix) and $det (A)=1$.  Suppose that $SO(3)$ acts on $\mathbb{S}^2_1$  with $g\to gx$. A random field $T$ is said to be $n$-weakly isotropic if $\mathbb{E}|T(gx)|^n < \infty$ ($n \geq 2$) for every $x \in \mathbb{S}^2_1$ and if, for every $x_1, \ldots , x_n \in \mathbb{S}^2_1$ and every $g \in SO(3)$ we have that
\begin{equation*}
\mathbb{E}[T(x_1) \times \cdots \times T(x_n)] = \mathbb{E}[T(gx_1) \times \cdots \times T(gx_n)].
\end{equation*}

 We consider the composition of the random field $T$  with an independent   time-changed  rotational Brownian motion $\mathfrak{B}^{x}_\nu(t) \in \mathbb{S}^2_1$, $t>t_0$ given by the time dependent random field
\begin{equation}
\left\lbrace \mathfrak{T}^{\nu}_t (x)=T(\mathfrak{B}^{x}_\nu(t)), \, x \in \mathbb{S}^2_1, \, t>t_0\right\rbrace, \quad \nu \in (0,1]. \label{QRFintro}
\end{equation}
We will call $ \{\mathfrak{T}^{\nu}_t (x), t>t_0\}$ random field on non-homogeneous sphere as the set $\{\mathfrak{B}^{x}_\nu(t), t>t_0\}\subset \mathbb{S}^2_1 $ is a non-homogeneous set.

 In our view, the composition \eqref{QRFintro} can be regarded as a random field with randomly shifted index. Indeed, we can introduce the shift $s_t$ such that
\begin{equation}
s_t\, x = x + \mathfrak{B}^{0}_\nu(t-t_0) = \mathfrak{B}^{x}_\nu(t), \quad \forall\, x \in \mathbb{S}^2_1 \label{shiftIntro}
\end{equation}
where $0=x_N$ is the North Pole. Thus, the time dependent random field \eqref{QRFintro} can be rewritten as follows
\begin{equation}
\left\lbrace \mathfrak{T}^{\nu}_t (x)=T(s_t\, x), \, x \in \mathbb{S}^2_1, \, t>t_0\right\rbrace, \quad \nu \in (0,1]. \label{QRFintroSchift}
\end{equation}
Formula \eqref{QRFintroSchift} says that the compact support of the random field $T$ is affected by the random action of the shift \eqref{shiftIntro}.

We first obtain in Lemma \ref{n-moment-time-changed-field} that
\begin{equation*}
\mathbb{E}[\mathfrak{T}^{\nu}_t (gx)]^n = \sum_{l_1\ldots l_n} \sqrt{\frac{(2l_1+1)\cdots (2l_n+1)}{(4\pi)^n}} \mathbb{E}[a_{l_10}\cdots a_{l_n 0}] < \infty, \quad g \in SO(3)
\end{equation*}
(is finite) where
$$a_{l_j m_j} = \int_{\mathbb{S}^2_1} T(x)Y^*_{l_j m_j}(x) \lambda(dx), \quad l_j \geq 0,\; |m_j | \leq l_j,\; j=1,2, \ldots , n$$ are the coefficient of $T(x)$ with respect to the orhonormal basis of spherical harmonics $\{Y_{lm}:l, m\in \mathbb{Z}, l\geq 0, |m|\leq l\}$ of $L^2(\mathbb{S}^2_1)$.
$a_{l_j m_j}$ are uncorrelated (over $l$) random variables, $\lambda$ is the Lebesgue measure on the sphere $\mathbb{S}^2_1$ and $\mathbb{E}[a_{l_10}\cdots a_{l_n 0}]$ is called the angular polyspectrum of order $n-1$ associated with the field $T$ (See Lemma 4.2 below). The fact that $\mathbb{E}[\mathfrak{T}^{\nu}_t (gx)]^n < \infty$ is a necessary condition for $\mathfrak{T}^{\nu}_t (x)$ to be isotropic.

The next  theorem  presents the time- and space- covariance of \eqref{QRFintro} and this is our second main result in this paper.
\begin{te}\label{second-main-thm}
For $0 \leq t_0 < t_1 \leq t_2 < \infty$ and $\nu \in (0,1]$ we have that:
\begin{itemize}
\item [i)] for $x\in \mathbb{S}^2_1$, $\forall\, g \in SO(3)$,
\begin{align}
\mathbb{E}[\mathfrak{T}^{\nu}_{t_0} (gx)\, \mathfrak{T}^{\nu}_{t_1}(gx) ] = & \sum_{l \geq 0} \frac{2l+1}{4\pi}C_{l}  \, \,  E_\nu(-\mu_l (t_1-t_0)^\nu)   \label{genCor1}
\end{align}
\item [ii)] for $x,y \in \mathbb{S}^2_1$ such that $x \neq y$,
\begin{align}
\mathbb{E}[\mathfrak{T}^{\nu}_{t_1} (x)\, \mathfrak{T}^{\nu}_{t_2}(y) ] = & \sum_{l \geq 0} \frac{2l+1}{4\pi}\, C_{l}\, \, \mathcal{E}_l(t_0, t_1, t_2) \, P_l(\langle x, y \rangle)  \label{genCor2}
\end{align}
where $C_l$ is the angular power spectrum defined in equation \eqref{propa2} and
\begin{equation*}
\mathcal{E}_l(t_0, t_1, t_2) = \prod_{i=1}^2 E_\nu(-\mu_l (t_i - t_{0})^\nu).
\end{equation*}
\end{itemize}
\end{te}

A stochastic process $X(t), \ t>0$ with $E(X(t))=0$ is said to have short range dependence if for fixed $t>0$,  $\sum_{h=1}^\infty E(X(t)X(t+h))<\infty$, otherwise it is said to have  long-range dependence.

\begin{coro}[Time-covariance]
For $h \geq 0$ and $\nu \in (0,1]$, the (equilibrium) time covariance of the random field $\mathfrak{T}^\nu_t(x)$, $t>t_0$ is given by
\begin{align}
\textsc{Cov}_{\mathfrak{T}}(h;\nu):=\mathbb{E}[\mathfrak{T}^\nu_{t_0+h}(x)\mathfrak{T}^\nu_{t_0}(x)]= & \sum_{l \geq 0} \frac{2l+1}{4\pi} C_{l} \, E_{\nu}(-\mu_l\,h^{\nu}), \quad \forall\, g \in SO(3). \label{timeCor}\\
\sim  &h^{-\nu}\sum_{l \geq 0} \frac{2l+1}{4\pi \mu_l} C_{l} , \, \mathrm{for} \ h\    \mathrm{large}. \label{asymptotic-covariance}
\end{align}
\end{coro}

\begin{os}
From equation \eqref{asymptotic-covariance} we see that random field $\mathfrak{T}^\nu_{t}(x)$ has long range dependence for $\nu\in (0,1)$ since
$$
\sum_{h=1}^\infty \textsc{Cov}_{\mathfrak{T}}(h;\nu)=\infty.
$$
\end{os}

\begin{coro}[Space-covariance]
For $x,y \in \mathbb{S}^2_1$ such that $x \neq y$, $t > t_0 \geq 0$ and $\nu \in (0,1]$, the space-covariance of the random field $\mathfrak{T}^\nu_t(x)$, $t>t_0$ is given by
\begin{equation}
\Gamma_{\nu}(\langle x,y \rangle, t) :=\mathbb{E}[\mathfrak{T}^{\nu}_t (x)\, \mathfrak{T}^{\nu}_t (y)]= \sum_{l \geq 0}  \frac{2l+1}{4\pi}\, C_l\,  \Big( E_{\nu}(-\mu_l (t - t_0)^\nu ) \Big)^2 \,P_l(\langle x, y \rangle). \label{spaceCor}
\end{equation}
\end{coro}

\begin{os}
Under suitable choice of the initial condition for the Cauchy problem involving the equation \eqref{pdeFracSphere}, for all $x \in \mathbb{S}^2_1$ and $t \to t_0$ we show that $\mathfrak{T}^{\nu}_t (x) \to T(x)$ and therefore we obtain the driving random filed on the homogeneous sphere as expected.
Indeed, for $t \to t_0$ we get that $\mathfrak{T}^{\nu}_t (x) \to T(x)$ only if the TRD has initial datum $u_0=\delta$. Indeed, we have that
\begin{equation*}
\lim_{t \to t_0} \mathbb{E}[\mathfrak{T}^{\nu}_t (x)\, \mathfrak{T}^{\nu}_t (y)] = \sum_{l \geq 0}  \frac{2l+1}{4\pi}\, C_l\,  \,P_l(\langle x, y \rangle)
\end{equation*}
which coincides with \eqref{TcovarianceFunction}.
\end{os}

\begin{os}
For $t \to \infty$ the Mittag-Leffler function goes to zero and thus, formula  \eqref{spaceCor} goes to the constant values
\begin{equation*}
\lim_{t \to \infty} \mathbb{E}[\mathfrak{T}^{\nu}_t (x)\, \mathfrak{T}^{\nu}_t (y)] = \frac{C_0}{4\pi}
\end{equation*}
where we have used the fact that $E_{\nu}(0)=1$,  and $P_0(\langle x, y \rangle)=1$. For the density of the TRD we get that $\lim_{t \to \infty} u_{\nu}(x, t; x_0, t_0) = 1/4\pi$ and thus, for all $\nu \in (0,1)$, we get a uniformly distributed r.v. on the whole sphere. We recall that $\lambda(\mathbb{S}^2_1)=4\pi$. The covariance of the random field $\mathfrak{T}^\nu_{\infty}(x)$ does not depend on the space and this means that
\begin{equation*}
\forall x \in S^2_1, \quad  \mathfrak{T}^\nu_{\infty}(x) \stackrel{d}{=} W(x), \quad \forall \nu \in (0,1)
\end{equation*}
where $\{W(x), \,x \in S^2_1\}$ is an uniformly distributed  noise on the sphere.
\end{os}

Our third main result is
\begin{te}\label{covariance-nu-1}
For $0 \leq t_0 < t_1 \leq t_2 < \infty$ and $\forall \, g \in SO(3)$, we have that
\begin{equation}\label{covariance-time-changed-field-bm}
\mathbb{E}[\mathfrak{T}^1_{t_1}(gx)\mathfrak{T}^1_{t_2}(gx)] = \sum_{l \geq 0} \frac{2l+1}{4\pi} C_l \, \, \exp(-\mu_l (t_2 - t_1)).
\end{equation}
\end{te}
  Theorem \ref{covariance-nu-1} is not a special case of Theorem \ref{second-main-thm}.In this case $\mathfrak{B}^{x}(t)=\mathfrak{B}^{x}_\nu(t)$
is a rotational Brownian motion on $\mathbb{S}^2_1$ which is Markovian. So we can prove more for this case.
\begin{os}
From equation \eqref{covariance-time-changed-field-bm} we see that random field $\mathfrak{T}^1_{t}(x)$ has short range dependence  since
$$
\sum_{h=1}^\infty \mathbb{E}[\mathfrak{T}^1_{t+h}(x)\mathfrak{T}^1_{t}(x)]= \sum_{h=1}^\infty  \sum_{l \geq 0} \frac{2l+1}{4\pi} C_l \, \, \exp(-\mu_l (h))<\infty
$$
when $C_l\sim l^{-\theta}$ as $\l\to\infty$ for $\theta>2$.
\end{os}
By simple conditioning, the independence of $\mathfrak{L}^{\nu}_t$ and $\mathfrak{B}^{1}(t)$, Theorem \ref{covariance-nu-1}, and Fubini theorem we have
\begin{equation}\begin{split}
\mathbb{E}[\mathfrak{T}^\nu_{t_1}(gx)\mathfrak{T}^\nu_{t_2}(gx)]  =&\int_{0}^\infty\int_{0}^\infty
\mathbb{E}[\mathfrak{T}^1_{z_1}(gx)\mathfrak{T}^1_{z_2}(gx)]H(dz_1,\ dz_2)\\
=&\int_{0}^\infty\int_{0}^\infty
\mathbb{E}[T(\mathfrak{B}^{gx}({z_1})T(\mathfrak{B}^{gx}({z_2})]H(dz_1,\ dz_2)\\
=&\sum_{l \geq 0} \frac{2l+1}{4\pi} C_l \, \,\int_{0}^\infty\int_{0}^\infty \exp(-\mu_l |z_1 - z_2|)H(dz_1,\ dz_2)\label{covariance-different-times}
\end{split}
\end{equation}
where $H(dz_1,\ dz_2)=\P(\mathfrak{L}^{\nu}_{t_1-t_0}\leq z_1, \mathfrak{L}^{\nu}_{t2-t_0}\leq z_2)$.
Using \eqref{covariance-different-times} and  Theorem 3.1 in \cite{meerschaert-skorski-correlation} we can deduce the following
\begin{coro}
For $0 \leq t_0 < t_1 \leq t_2 < \infty$ and $\forall \, g \in SO(3)$, denote by $T_1=t_1-t_0$ and $T_2=t_2-t_0$ we have that
\begin{equation}\label{covariance-time-changed-field-bm}
\mathbb{E}[\mathfrak{T}^\nu_{t_1}(gx)\mathfrak{T}^\nu_{t_2}(gx)] = \sum_{l \geq 0} \frac{2l+1}{4\pi} C_l \, \, \bigg[E_\nu (-\mu_lT_2^\nu )+\mu_l\nu T_2^\nu \int_{0}^{T_1/T_2}\frac{E_\nu(-\mu_lT_2^\nu (1-z)^\nu)}{z^{1-\nu}} dz\bigg].
\end{equation}
\end{coro}
\begin{os}
Since  by Remark 3.3 in \cite{meerschaert-skorski-correlation} we have  for fixed $T_1=(t_1-t_0)>  0$
$$
E_\nu (-\mu_lT_2^\nu )+\mu_l\nu T_2^\nu \int_{0}^{T_1/T_2}\frac{E_\nu(-\mu_lT_2^\nu (1-z)^\nu)}{z^{1-\nu}} dz\sim
\frac{1}{T_2^\nu \Gamma(1-\nu )}\bigg(\frac{1}{\mu_l}+\frac{T_1^\nu}{\Gamma(1+\nu)}\bigg)
$$
as $T_2=(t_2-t_0)\to\infty$,
we deduce that
$\mathfrak{T}^\nu_{t_1}(x)$ has long-range dependence.
\end{os}

\paragraph*{Motivations}
 The  model  we present describe a random motion over a random surface. Indeed, if we write the TRD, with staring point $(\vartheta_0, \varphi_0)$ at time $t_0$, as follows
$$ ( \vartheta_t , \varphi_t ) = \mathfrak{B}^{(\vartheta_0, \varphi_0)}_\nu(t), \quad t>0$$
then we have that
$$\left( \vartheta_t, \varphi_t, T(\vartheta_t, \varphi_t) \right) \in \mathbb{R}^3$$
is a point representing a randomly moving particle over a random surface and therefore a motion on a random environment.Apart from the mathematical interest in studying time-dependent random fields indexed by a random environment, we want to provide a new random field in which the anisotropies of the CMB radiation can be explained. Indeed, since  the set $\{\mathfrak{B}^{x}_\nu(t), t>t_0\}\subset \mathbb{S}^2_1 $ is  a non-homogeneous subset of the sphere, in our view, the field $\mathfrak{T}^\nu_t$ captures the anisotropies by which the CMB radiation is affected. Furthermore, the multiparameter process we present well explain also the observational error due to instruments and depending on the observation time. Indeed, the shift representation \eqref{QRFintroSchift} can be considered as a process which well describe such observational error. As time passes, we are not observing $T(x)$ but $T(x +$TRD$)$ where the TRD is a noise depending on time. A further remarkable feature is that the new random field $\mathfrak{T}^\nu_t(x)$ possesses an angular power spectrum, say $C_l(t)$ (depending on time), which goes from an exponential (for $\nu \in  (0,1)$) to a polynomial (for $\nu \to 0$) behaviour as the frequency $l \to \infty$. This result plays an important role in the asymptotic theory  in the high-frequency sense. If the random field is Gaussian for instance, then its dependence structure is completely identified by the angular correlation function and the angular power spectrum. For non-Gaussian fields, we need  higher-order correlation functions in order to characterize the dependence structure and, in turn  higher-order angular power spectra, polyspectra. The asymptotic behavior of ($m$-points) higher-order moments of a random field is therefore of great interest in studying the asymptotic nature of the random field, as asymptotic Gaussianity for instance. In our case, we do not consider Gaussianity but the high-resolution analysis of $\mathfrak{T}^\nu_{l,t}$ reveals different asymptotic covariance structure for such a random field, depending on $\nu \in (0,1)$ as pointed  in Remark  \ref{RemarkShortLongDep}. In real data, we get more and more information (or resolution) as $l$ increases.

\paragraph*{Overview of the paper}
The plan of the work is as follows. We introduce fundamental concepts on fractional calculus  in Section \ref{secSSIP}.  In Section \ref{SecRBM} we study time-changed diffusions on the sphere and  establish the connection of the time-changed Brownian motion on the sphere to fractional order PDEs involving spherical Laplacian. In Section \ref{SecRFNHS}, we introduce random fields on the sphere and define a new multiparameter process which is the time dependent random field on spherical non-homogeneous surface.

\paragraph*{Notations} For the reader's convenience we list below some useful symbols:
\begin{itemize}
\item $T(x)$ is the isotropic random field on the (homogeneous) sphere,
\item $\mathfrak{H}^\nu_t$ is the stable subordinator with density $h_\nu$,
\item $\mathfrak{L}^\nu_t$ is the inverse to a stable subordinator with density $l_\nu$,
\item $B^{x_0}(t)$ is the Brownian motion on the sphere (rotational Brownian motion) starting from $x_0 \in \mathbb{S}^2_1$ at $t_0 \geq 0$,
\item $\mathfrak{B}^{x_0}_{\nu}(t) = B^{x_0}(\mathfrak{L}^\nu_t)$ is the time-changed rotational Brownian motion (time-changed rotational diffusion, TRD) with law $u_\nu$,
\item $\mathfrak{T}^{\nu}_t (x) = T(\mathfrak{B}^{x}_{\nu}(t))$ is the time dependent random field on the non-homogeneous sphere.
\end{itemize}
Furthermore,
\begin{itemize}
\item $D^\nu_t$ or $\partial^\nu_t$ with $\nu \in (0,1)$ is the Dzhrbashyan-Caputo  fractional derivative of order $\nu\in(0,1)$,
\item $\mathbb{D}^\nu_t$ with $\nu \in (0,1)$ is the Riemann-Liouville fractional derivative of order $\nu\in(0,1)$,
\end{itemize}
and, for $\nu=1$
$$D^1_t ={\mathbb{D}^1_t}=d/d t.$$

\section{Inverse stable subordinators and Mittag-Leffler function}
\label{secSSIP}

A stable subordinator $\mathfrak{H}^{\nu}_t$, $t>0$, $\nu \in (0,1)$,  is (see \cite{Btoi96}) a L\'evy process with non-negative, independent and stationary increments with Laplace transform in \eqref{lapH}


The inverse stable subordinator  $\mathfrak{L}^{\nu}$ defined in \eqref{inverse-stable} with density, say $l_\nu$, satisfies
\begin{equation}
Pr\{ \mathfrak{L}^{\nu}_t < x \} = Pr\{ \mathfrak{H}^{\nu}_x > t \}. \label{relPHL}
\end{equation}
Let $\mathbb{D}^\nu _t g(t)$ be Riemann-Liouville fractional derivative of order $\nu\in (0, 1)$ defined by its Laplace transform $\int_0^\infty  e^{-st}\mathbb{D}^\nu _t g(t) dt=s^\nu \tilde g(s)$.
According to \cite{BM01, Dov4, meerschaert-straka}, represents a stochastic solution to
\begin{equation*}
\left( \mathbb{D}^{\nu}_{t} + \frac{\partial}{\partial x} \right) l_{\nu}(x,t)=0, \quad x > 0\; ,t>0,\, \nu \in (0,1)
\end{equation*}
subject to the initial and boundary conditions
\begin{equation}
\left\lbrace \begin{array}{l} l_{\nu}(x,0) = \delta(x), \quad x>0,\\ l_{\nu}(0,t) = t^{-\nu}/\Gamma(1-\nu), \quad t>0. \end{array} \right .\label{fracinicond}
\end{equation}
Due to the fact that $\mathfrak{H}^\nu_t$, $t>0$ has non-negative increments, that is non-decreasing paths, we have that $\mathfrak{L}^\nu_t$ is a hitting time. Furthermore, for $\nu \to 1$ we get that
\begin{equation*}
\lim_{\nu \to 1} \mathfrak{H}^\nu_t = t = \lim_{\nu \to 1} \mathfrak{L}^\nu_t
\end{equation*}
almost surely (\cite{Btoi96}) and therefore $t$ is the elementary subordinator. From the relation \eqref{relPHL} and the Laplace transform \eqref{lapH}, after some algebra, we arrive at the Laplace transform of $\mathfrak{L}^\nu$  given by \eqref{lapL}.

Next we state some of the properties of the Mittag-Leffler function. Let $\nu\in (0,1]$
As we can immediately check $E_{\nu}(0)=1$ and (see for example \cite{KST06, pod99} )
\begin{equation}
0 \leq E_{\nu}(-z^\nu) \leq \frac{1}{1+z^\nu} \leq 1, \quad z \in [0, +\infty). \label{Ebound}
\end{equation}
Indeed, we have that
\begin{equation}
E_\nu(-z^\nu) \approx 1- \frac{z^\nu}{\Gamma(\nu +1)}  \approx \exp\left( - \frac{z^\nu}{\Gamma(\nu +1)} \right), \quad 0<z \ll 1
\end{equation}
whereas,
\begin{equation}\label{mittag-leffler-large-asymptotics}
E_\nu(-z^\nu) \approx \frac{z^{-\nu}}{\Gamma(1-\nu)}  - \frac{z^{-2\nu}}{\Gamma(1-2\nu)}+ \ldots , \quad z \to +\infty.
\end{equation}
Thus the Mittag-Leffler function is a stretched exponential with heavy tails. Furthermore, we have that (see \cite[formula 2.2.53]{KST06})
\begin{equation}
\mathbb{D}_{z}^{\nu}\,  E_{\nu}(\mu z^{\nu}) = \frac{z^{- \nu}}{\Gamma(1 -\nu)} + \mu E_{\nu}(\mu z^\nu), \quad \mu \in \mathbb{C}. \label{DfracE2}
\end{equation}
Dzhrbashyan-Caputo   derivative of order $\nu\in (0,1)$ is  defined  in \cite{Caputo} by
\begin{equation}
{D^{\nu}_{z}} f(z) = \frac{1}{\Gamma(1 -\nu)} \int_0^z \frac{d f(s)}{d s} \, (z-s)^{ - \nu } \, ds. \label{capFracDer}
\end{equation}
  Formula \eqref{capFracDer} can be also written in terms of the Riemann-Liouville derivative as follows (see \cite{KST06, SKM93})
\begin{equation}
{D^{\nu}_{z}} f(z) = \mathbb{D}^\nu_{z} f(z)  -  f(z) |_{z=0^+} \frac{z^{- \nu}}{\Gamma(1 - \nu )}, \label{RCfracder}
\end{equation}
and therefore, formula \eqref{DfracE2} takes the form
\begin{equation}
{D^{\nu}_{z}}\,  E_{\nu}(\mu z^{\nu}) = \mu E_{\nu}(\mu z^\nu), \quad \mu \in \mathbb{C}, \quad \nu \in (0,1). \label{eigenDcaputo}
\end{equation}
Hence in this case we say that  $E_{\nu}(\mu z^{\nu})$ is the eigenfunction of the  Dzhrbashyan-Caputo derivative operator ${D^{\nu}_{z}}$ with   the corresponding eigenvalue $\mu\in \mathbb{C}$.

\section{Time-changed Rotational Brownian motion on the  sphere}
\label{SecRBM}
In this section we define a measurable map from the probability space $(\Omega, \mathfrak{F}_{\mathfrak{B}}, P)$ to the measurable space $(\mathbb{S}^2, \mathcal{B}(\mathbb{S}^2), \lambda_{\mathfrak{B}})$ which is the time-changed rotational Brownian motion $\mathfrak{B}^{x_0}_{\nu}(t)$, $t>t_0$, $\nu \in (0,1]$ with starting point $x_0 \in \mathbb{S}^2_1$ at time $t_0 >0$. Such a process can be regarded as a time-changed rotational diffusion (TRD) or a rotational Brownian motion on the sphere $\mathbb{S}^2_1$  time-changed by an inverse to a stable subordinator. Thus, the composition we deal with is written as $\mathfrak{B}^{x_0}_{\nu}(t)=B^{x_0}(\mathfrak{L}^{\nu}_t)$, $t>t_0$ where $B^{x_0}$ is a rotational Brownian motion starting from $x_0 \in \mathbb{S}^2_{1}$ at $t_0 \geq 0$ and the time-change is given by $\mathfrak{L}^\nu_t$, $t>0$ which is the inverse to a stable subordinator of index $\nu\in (0,1)$. The TRD represents the spherical counterpart of the fractional diffusion on bounded domain driven by the fractional equation
\begin{equation}
\begin{cases}
(D^{\nu}_{t} - \triangle) u(\mathbf{x}, t)=0,\quad \mathbf{x} \in \mathbb{R}^n,\; t>0\\
u(\mathbf{x}, 0)=u_0(\mathbf{x})
\end{cases}
\label{unboundCP}
\end{equation}
where $\triangle=\sum_{i=1}^n\partial^2_{x_i}$ is the Laplace operator and whose solutions can be written in terms of Wright or Fox functions (\cite{Dov2, FOX61,OB09}) and by using time-changed Brownian motion in $\mathbb{R}^n$; see for example \cite{MSheff04, BMN09ann}. In \cite{OB09} the authors presented the explicit one dimensional solutions for some particular order $\alpha$ of the fractional derivative. For information on fractional diffusions on unbounded domains the reader can consult, for example,  the works  \cite{Koc90, MLP01,Wyss86, SWyss89}.  In  \cite{BM01} the fractional Cauchy problem involving an infinitely divisible generator on a finite dimensional space has been investigated while \cite{BMN09ann} studied the fractional Cauchy problem \eqref{unboundCP} in a bounded domain $\mathcal{D} \subset \mathbb{R}^n$.   In \cite{BMN09ann}, the authors found that (in our notation) the strong solution of \eqref{unboundCP} in a bounded domain $\mathcal{D}$ with Dirichlet boundary conditions is given by
\begin{equation*}
u(\mathbf{x}, t) = \mathbb{E} u_0(\mathbf{B}^{\mathbf{x}}(\mathfrak{L}^\nu_t)) \, 1_{(\tau_{_\mathcal{D}}(\mathbf{B})> \mathfrak{L}^\nu_t)}, \quad \mathbf{x} \in \mathcal{D},\; t>0
\end{equation*}
where
\begin{equation*}
\tau_{_\mathcal{D}}(\mathbf{B}) = \inf \{ s \geq 0:\, \mathbf{B}(s) \notin \mathcal{D} \}
\end{equation*}
is the first exit time of Brownian motion  $\mathbf{B}$ from $\mathcal{D}\subset \RR{R}^n$.


Our aim, in this section, is to study the solution to the fractional Cauchy problem involving the spherical Laplacian operator  which extends the results in \cite{meerschaert-skorski, BMN09ann } and prove Theorem \ref{TheoremMain}.
\begin{os}\label{strong-caputo-derivative}

We say that  $\Delta_{\mathbb{S}_{1}^2}u$
 exists in the strong sense if it exists pointwise and is continuous in $D$.

 Similarly, we say that $\partial^\nu_t f(t)$ exists in the strong sense if it exists pointwise and is continuous for  $t\in [0,\infty)$. One sufficient condition for this is the fact $f$ is a $C^1$ function on $[0, \infty)$ with
 $|f'(t)| \leq c \, t^{\gamma -1}$ for some $\gamma >0$.
 Then by \eqref{capFracDer}, the Caputo fractional derivative
 $\partial^\nu_t f(t)$ of $f$ exists
 for every $t>0$ and the derivative is continuous in $t>0$.

\end{os}

It is well-known that (\cite{karlin-taylor})
the solutions to the eigenvalue problem
\begin{equation}
\triangle_{\mathbb{S}_{1}^2} u(x)= - \mu u(x), \ \ x\in \mathbb{S}_{1}^2  \label{eigenY}
\end{equation}
is solved with  a sequence of  eigenvalues
$\mu_l= l(l+1)$, $l\in \mathbb{Z}, l\geq 0$ with their corresponding eigenfunctions given by the spherical harmonics (here we use the fact that $x\in \mathbb{S}_{1}^2$ can be represented as $x=(\sin\vartheta \cos \varphi, \sin\vartheta \sin \varphi, \cos \vartheta )$)
\begin{equation}\label{spherical-harmonics}
Y_{lm}(x)=Y_{lm}(\vartheta, \varphi) = \sqrt{\frac{2l+1}{4 \pi} \frac{(l-m)!}{(l+m)!}} P_{lm}(\cos \vartheta) e^{im\varphi}, \quad |m|\leq l, \quad \vartheta \in [0,\pi],\; \varphi \in [0, 2\pi]
\end{equation}
(or linear combination of them) where
\begin{equation*}
P_{lm}(z)=(-1)^m (1-z^2)^{m/2}\frac{d^m}{d z^m}P_{l}(z)
\end{equation*}
are the associated Legendre functions, and  the well-known Legendre polynomials $P_l$ are defined by the Rodrigues' formula
$$ P_l(z) = \frac{1}{2^l l!}\frac{d^l}{dz^l} (z^2 - 1)^l. $$

\begin{os}
It is important to note here that the operator considered in  Theorem 3.2 in  \cite{meerschaert-skorski} for the Jacobi case  (the case $a=b=0$ in their notation)is different from equation \eqref{pdeFracSphere} and \eqref{eigenY} here.  First, their fractional Pearson diffusion equation is in $\mathbb{R}$ whereas  our equation is in $\mathbb{S}_{1}^2$. Second, the spherical harmonics in \eqref{spherical-harmonics} are different from the eigenfunctions called Jacobi polynomials in their notation. Jacobi polynomials are the eigenfunctions corresponding to the eigenvalue problem
$$
(1-x^2)\frac{d^2g(x)}{dx^2}-2x\frac{dg(x)}{dx}=-\lambda g(x), x\in \mathbb{R}
$$ and it has eigenfunctions
$g_l(x)=(-1)^l2^l l!P_l(x)$ and corresponding eigenvalues $\lambda_l=l(l+1).$

\end{os}
We list some properties which will turn out to be useful further in the text: for all $x \in \mathbb{S}^2_1$ (symmetry) we have that
\begin{equation}
Y_{lm}^{*}(x) = (-1)^m Y_{l-m}(x), \label{Yconj}
\end{equation}
and in spherical coordinates $Y^*_{lm}(\vartheta, \varphi) = Y_{lm}(\vartheta, - \varphi)$ where $*$ stands for the complex conjugation; for all $l_1,l_2,m_1,m_2$ (orthonormality) we also write the formula \eqref{appendixInt2Y} as
\begin{align}
\int_{\mathbb{S}^2_1} Y_{l_1m_1}(x)\, Y_{l_2m_2}^{*}(x)\, \lambda(dx) = \delta_{l_1m_1}^{l_2m_2}  \label{Yortho}
\end{align}
where $\delta_{l_1m_1}^{l_2m_2} = \delta_{l_1}^{l_2}\, \delta_{m_1}^{m_2}$ are the Kronecker's delta symbols \eqref{kronSy} and $\lambda(\cdot)$ is the Lebesgue measure on $\mathbb{S}^2_1$ given in spherical coordinates by $$\lambda(dx)=\sin \vartheta\, d\vartheta\, d\varphi\ \  \mathrm{for}\ \ x=(\sin\vartheta \cos \varphi, \sin\vartheta \sin \varphi, \cos \vartheta );$$
for all $x, y \in \mathbb{S}^2_1$ (addition formula) we have that (see page 339 in \cite{karlin-taylor} for this and other properties of spherical harmonics)
\begin{equation}
\sum_{m=-l}^{+l} Y_{lm}(x) \, Y_{lm}^{*}(y) = \frac{2l+1}{4\pi} P_l(\langle x, y \rangle) \label{additionT}
\end{equation}
where $\langle x, y \rangle = \cos d( x, y)$ and $d(x, y)$ is the usual spherical distance  on $\mathbb{S}^2_1$.  Furthermore,
\begin{equation}\label{symmetry-equality}
\sum_{m=-l}^{+l} Y_{lm}(x) \, Y_{lm}^{*}(x) = \frac{2l+1}{4\pi}
\end{equation}
and $P_l(\langle x , x \rangle) = 1$; for all $x, y \in \mathbb{S}^2_1$ (reproducing kernel) the following holds
\begin{equation}
 \int_{\mathbb{S}^2} P_{l_1}(\langle x,z \rangle)\, P_{l_2}(\langle z, y \rangle)\, \lambda(dz) = \delta_{l_1}^{l_2} \frac{4\pi}{2l_1+1} P_{l_1}(\langle x, y \rangle). \label{reproducK}
\end{equation}

Let  $f \in L^2(\mathbb{S}^2_1)$.  Since $\{Y_{lm}:l\in \mathbb{N},  m\in \mathbb{Z} \ \mathrm{and}\ |m|\leq l\}$ is dense in $L^2(\mathbb{S}^2_1)$. The function $f$ (see also the Peter-Weyl theorem on the sphere in \cite{MarPeccBook} and the references therein) can be written as
\begin{equation*}
f(x)=f(\vartheta, \varphi) = \sum_{l=0}^\infty \sum_{m=-l}^{+l} a_{lm}(f) Y_{lm}(\vartheta, \varphi)
\end{equation*}
which holds in the $L^2$ sense, where
\begin{equation*}
a_{lm}(f)=\int_{\mathbb{S}^2_1} f(\vartheta, \varphi) Y^*_{lm}(\vartheta, \varphi)\, \sin \vartheta\, d\vartheta\, d\varphi .
\end{equation*}
We define
\begin{equation}\label{def-angular-spectrum}
A_l (f) = \sum_{m=-l}^{+l}| a_{lm}(f)|^2
\end{equation}
as the {\bf angular power spectrum} of $f$. The angular power spectrum can be considered as a power law of index, say $\varpi$, which can be different from zero (red spectrum for $\varpi < 0$ and blue spectrum for $\varpi >0$) or not (white spectrum). If $A_l(f)$  decays exponentially as $l \to \infty$, then $f$ is a real and analytic function and turns out to be infinitely differentiable on the sphere. In general, for $s\geq 0$, we introduce the Sobolev spaces
\begin{equation}
H^{s,2}(\mathbb{S}^2_1) = \left\lbrace f \in L^2(\mathbb{S}^2_1):\, \| f\|_{s, 2} = \sum_{l=0}^{\infty} (2l+1)^{2s} A_l (f) < \infty  \right\rbrace. \label{SobolevS}
\end{equation}
 The Sobolev space \eqref{SobolevS} is the closure of the set of all spherical polynomials with respect to the norm $\| \cdot \|_{s, 2}$. Furthermore, $H^{s,2}(\mathbb{S}^2_1)$ is a Hilbert space, with inner product
\begin{equation*}
(f_1,f_2)_{s,2} = \sum_{l= 0}^\infty (2l+1)^{2s} \sum_{m=-l}^{+l} a_{lm}(f_1)\, a_{lm}^*(f_2), \quad f_1,f_2 \in H^{s,2}(\mathbb{S}^2_1).
\end{equation*}
Obviously, $H^{0,2}(\mathbb{S}^2_1) \subset L^2(\mathbb{S}^2_1)$ and, for $s^\prime > s$, we get that
\begin{equation}
f \in H^{s^\prime ,2}(\mathbb{S}^2_1) \quad \Rightarrow \quad f \in H^{s,2}(\mathbb{S}^2_1).
\end{equation}
We write $H^{s}(\mathbb{S}^2_1)$ instead of $H^{s,2}(\mathbb{S}^2_1)$ and say that $f \in H^{s}(\mathbb{S}^2_1)$ is infinitely differentiable provided that $A_l \approx l^{-2s-\theta}$ with $\theta > 1$ as $l \to \infty$ ($s \geq 0$)  by the Sobolev embedding Theorem.

We are now ready to give the

\begin{proof}[{\bf Proof of Theorem \ref{TheoremMain}}]

 The proof follows the main steps in the proof of Theorem 3.2 in \cite{meerschaert-skorski} and Theorem 3.1 in \cite{BMN09ann}. Since operator $\Delta\mathbb{S}^2_1$ is quite different operator  than those in these two papers, we give all the details of the proofs.
The proof is based on the method of separation of variables.  Let $u(t,x)=G(t)F(x)$ be a solution of
(\ref{pdeFracSphere}).
 Then substituting into  (\ref{pdeFracSphere}), we get
$$
F(x)\partial_t^\nu G(t)
= G(t)\triangle_{\mathbb{S}_{1}^2}  F(x).
$$
Divide both sides by $G(t)F(x)$ to obtain
$$
\frac{\partial^\nu_t G(t)}{G(t)} = \frac{\triangle_{\mathbb{S}_{1}^2}  F(x)}{F(x)}=-\mu.
$$
Then we have
\begin{equation}\label{time-pde}
\partial^\nu_t G(t)=-\mu G(t), \ t>0
\end{equation}
and
\begin{equation}\label{space-pde}
\triangle_{\mathbb{S}_{1}^2} F(x)=-\mu F(x), \ x\in \mathbb{S}_{1}^2.
\end{equation}
By the discussion above, the eigenvalue problem (\ref{space-pde}) is solved by an infinite
sequence of pairs $\{\mu_{lm}=l(l+1), Y_{lm}:\, l \geq 0,\, -l \leq m \leq +l \}$ where  the spherical harmonics $Y_{lm}$ forms a complete orthonormal set in $L^2(\mathbb{S}_{1}^2)$.
In particular, the initial function $u_0$ regarded as an element of $L^2(\mathbb{S}_{1}^2)$ can be represented as
\begin{equation}
u_0(x-x_0)=\sum_{l=0}^\infty\sum_{m=-l}^l b_{lm}(x_0)Y_{lm}(x).
\end{equation}
where $b_{lm}(x_0)=\int_{\mathbb{S}_{1}^2}u_0(x-x_0)Y^* _{lm}(x)\lambda(dx)=R_l Y^*_{lm}(x_0).$
By equation \eqref{eigenDcaputo} we  see that $a_{lm}(t;x_0,t_0):= b_{lm}(x_0)E_\nu(-\mu_l(t-t_0)^\nu)$
solves \eqref{time-pde}.  Sum these solutions $ b_{lm} E_\nu(-\mu_l(t-t_0)^\nu) Y_{lm}(x)$ to (\ref{pdeFracSphere}), to get
\begin{equation}\label{formal-sol-L-1}
\begin{split}
u_{\nu}(x, t; x_0, t_0)&=\sum_{l=0}^\infty\sum_{m=-l}^lb_{lm}(x_0)E_\nu(-\mu_l(t-t_0)^\nu) Y_{lm}(x)=\sum_{l=0}^\infty\sum_{m=-l}^la_{lm}(t;x_0,t_0) Y_{lm}(x)\\
&=\sum_{l = 0}^{\infty} \sum_{m=-l}^{+l} R_l\, Y^*_{lm}(x_0)\, E_{\nu}(-\mu_l\, (t - t_0)^\nu)\, Y_{lm}(x).
\end{split}
\end{equation}
By the addition formula \eqref{additionT}, we can write
\begin{equation}\label{formal-l-sol-no-m}
u_{\nu}(x, t; x_0, t_0) = \sum_{l \geq 0} R_l\, E_{\nu}\left( - \mu_l\, (t-t_0)^{\nu} \right)  \frac{2l+1}{4\pi} P_l(\langle x, x_0 \rangle).
\end{equation}
 This is similar to the form of the solution  of a Factional Jacobi diffusion that was worked in  section 7.4 and 7.5 in \cite{meerschaert-skorski-book} for the case $a=b=0$ in their notation.

It remains to show that \eqref{formal-sol-L-1} solves \eqref{pdeFracSphere} and satisfies the conditions of Theorem \ref{TheoremMain}. Since $u_0\in H^s(\mathbb{S}^2_1)$ we know that
\begin{equation}\label{initial-angular-spectrum}
A_l(u_0)=\sum_{m=-l}^l (b_{lm}(x_0))^2=\sum_{m=-l}^l (R_l Y^*_{lm}(x_0))^2\leq C (2l+1)^{-2s}
\end{equation}
where $C=||u_0||_{s,2}.$ The angular power spectrum  for $u_{\nu}(x, t; x_0, t_0)$ in  \eqref{formal-sol-L-1} is  then given by
\begin{equation}\label{angular-spectrum-solution}
\begin{split}
A_l(t, t_0) = & \sum_{|m|\leq l} \big|a_{lm}(x; x_0, t_0) \big|^2\\
= & \Big( R_l \, E_{\nu}(-\mu_l\, (t - t_0)^\nu) \Big)^2 \sum_{|m|\leq l} \big| Y^*_{lm}(x_0) \big|^2\\
=&\Big( \, E_{\nu}(-\mu_l\, (t - t_0)^\nu) \Big)^2 A_l(u_0)\\
\leq & A_l(u_0).
\end{split}
\end{equation}
Hence for fixed $t>t_0$ as $l\to \infty$, using \eqref{mittag-leffler-large-asymptotics} we can see that
\begin{equation}\label{angular-spectrum-asymptotic}
A_l(t, t_0)\sim (\mu_l(t-t_0)^\nu)^{-2}A_l(u_0)\sim \bigg(\frac{1}{(t-t_0)^\nu\Gamma(1-\nu)}\bigg)^{2}l^{-4}A_l(u_0).
\end{equation}

{\bf Step 1. }First we check that
$$\| u_{\nu}(\cdot, t; x_0, t_0)\|^2_{L^2(\mathbb{S}^2_1)} < \infty$$
uniformly in $t \geq t_0$. First we notice that
 From the Parseval's identity, Equation \eqref{angular-spectrum-solution} and  Equation \eqref{Ebound} we get that
\begin{align*}
\| u_{\nu}(\cdot, t; x_0, t_0)\|^2_{L^2(\mathbb{S}^2_1)} = & \sum_{lm} \big| a_{lm}(t; x_0,t_0) \big|^2 = \sum_{l \geq 0} A_l(t, t_0)\\
\leq & \sum_{l \geq 0} A_l(u_0)=\| u_0(\cdot - x_0) \|^2_{L^2(\mathbb{S}^2_1)}
\end{align*}
 this proves that $u_\nu \in L^2(\mathbb{S}^2_1)$.

{\bf Step 2.}
Next we show that the series
\eqref{formal-sol-L-1}
is  absolutely and uniformly convergent for $t>t_0
>0.$

Using  \eqref{Ebound}, \eqref{initial-angular-spectrum},  \eqref{additionT} and Cauchy-Schwartz inequality we have
\begin{eqnarray}
& &\sum_{l = 0}^{\infty} \sum_{m=-l}^{+l} |a_{lm}(t; x_0, t_0)|\, |Y_{lm}(x)|\nonumber\\
&\leq &
\sum_{l = 0}^{\infty} \sum_{m=-l}^{+l} |b_{lm}(x_0)| E_\nu(\mu_l(t-t_0)^\nu) \, |Y_{lm}(x)|\nonumber\\
&\leq &\sum_{l = 0}^{\infty}   \sqrt{\sum_{m=-l}^{+l} (E_\nu(\mu_l(t-t_0)^\nu))^2|b_{lm}(x_0)|^2 } \sqrt{\sum_{m=-l}^l ( Y_{lm}(x))^2}\nonumber\\
&\leq &(t-t_0)^{-\nu}\sum_{l = 0}^{\infty}  l^{-2} \sqrt{C} (2l+1)^{-s} \sqrt{\frac{(2l+1)}{4\pi}}<\infty\label{absolute-bound}
\end{eqnarray}
since $s>0$ we have $s+2-1/2=s+3/2>1$.
Hence the function in \eqref{formal-sol-L-1}  is absolutely and uniformly convergent.

{\bf Step 3.} We next show that $\partial_t^\nu u_{\nu}(x, t; x_0, t_0)$  exists pointwise as a continuous function.
Using \cite[Equation (17)]{krageloh}
$$
\left| \frac{ d E_\nu(-\mu_l
(t-t_0)^\beta)}{dt}\right|\leq c\frac{\mu_l( t-t_0)^{\nu-1}}{1+\mu_l (t-t_0)^\nu}\leq c\mu_l(t-t_0)^{\nu -1},
$$
and  \eqref{initial-angular-spectrum} and \eqref{symmetry-equality} we get
\begin{eqnarray}
\left|\partial_t u_{\nu}(x, t; x_0, t_0)\right|&\leq & \sum_{n=1}^\infty |b_{lm}(x_0)|\left| \frac{ d E_\nu(-\mu_l
(t-t_0)^\nu)}{dt}\right||Y_{lm}(x)| \nonumber\\
&\leq &c(t-t_0)^{\nu -1} \sum_{l = 0}^{\infty} \sum_{m=-l}^{+l} \mu_l|b_{lm}(x_0)||Y_{lm}(x)|\nonumber\\
&:=&c(t-t_0)^{\nu -1}g(x)\nonumber\\
&\leq &c(t-t_0)^{\nu -1} \sum_{l = 0}^{\infty}\mu_l\sqrt{ \sum_{m=-l}^{+l} |b_{lm}(x_0)|^2}\sqrt{\sum_{m=-l}^n|Y_{lm}(x)|^2}\nonumber\\
&\leq &c\sqrt{C/4\pi}(t-t_0)^{\nu -1} \sum_{l = 0}^{\infty}  (l(l+1))(2l+1)^{-s+1/2}<\infty\nonumber
\end{eqnarray}
since $-s+5/2<-1$  when $s>7/2=1+5/2$. Hence by Remark \ref{strong-caputo-derivative}, $\partial_t^\nu u_{\nu}(x, t; x_0, t_0)$  exists pointwise as a continuous function and  is defined as a classical function.

{\bf Step 4.}
 We show here that we can
 apply $\Delta_{\mathbb{S}^2_1}$ and $\partial^\nu_t$ term-by-term to the series \eqref{formal-sol-L-1}. For this  we need to show that the series
\eqref{newserEquals}
is  absolutely and uniformly convergent for $t>t_0
>0.$

From \eqref{eigenY}, term by term we get
\begin{equation}
\begin{split}
\sum_{l = 0}^{\infty} \sum_{m=-l}^{+l}\, \partial^{\nu}_{t}\, a_{lm}(t; x_0, t_0) Y_{lm}(x)  &= - \sum_{l = 0}^{\infty} \sum_{m=-l}^{+l} a_{lm}(t; x_0, t_0) \mu_l\, Y_{lm}(x) \label{newserEquals}\\
&=\sum_{l \geq 0} \mu_l\, R_l\, E_{\nu}\left( - \mu_l\, (t-t_0)^{\nu} \right) \frac{2l+1}{4\pi} P_l(\langle x, x_0 \rangle)
\end{split}
\end{equation}

 By the previous step, using \eqref{Ebound}, \eqref{symmetry-equality} and the Cauchy-Schwartz inequality
 we get
\begin{eqnarray}
& &\sum_{l = 0}^{\infty} \sum_{m=-l}^{+l} |a_{lm}(t; x_0, t_0)| \mu_l\, |Y_{lm}(x)|\nonumber\\
&\leq &
\sum_{l = 0}^{\infty} \sum_{m=-l}^{+l} |b_{lm}(x_0)| E_\nu(\mu_l(t-t_0)^\nu) \mu_l\, |Y_{lm}(x)|\nonumber\\
&\leq &\sum_{l = 0}^{\infty} |E_\nu(\mu_l(t-t_0)^\nu) \mu_l \sqrt{\sum_{m=-l}^{+l} |b_{lm}(x_0)|^2 } \sqrt{\sum_{m=-l}^l ( Y_{lm}(x))^2}\nonumber\\
&\leq &\sum_{l = 0}^{\infty}  \frac{\mu_l}{1+\mu_l (t-t_0)^\nu} \sqrt{C} (2l+1)^{-s} \sqrt{\frac{(2l+1)}{4\pi}}\label{absolute-bound}\\
&\leq & (t-t_0)^{-\nu}\sum_{l = 0}^{\infty} \sqrt{C/4\pi}(2l+1)^{-s+1/2}<\infty
\end{eqnarray}
since $s>3/2$ we have $s-1/2>1$.
Hence
 $$
\sum_{l = 0}^{\infty} \sum_{m=-l}^{+l}\, \partial^{\nu}_{t}\, a_{lm}(t; x_0, t_0) Y_{lm}(x)
=\sum_{l = 0}^{\infty} \sum_{m=-l}^{+l}\, \, a_{lm}(t; x_0, t_0) \Delta_{\mathbb{S}^2_1}Y_{lm}(x)
$$
since the two series are equal term-by-term, and since the series on the right converges absolutely.

Now it is easy to check that
 the fractional time derivative and $\Delta_{\mathbb{S}^2_1}$ can be applied term by term in
(\ref{formal-sol-L-1}) to give
\begin{equation*}\begin{split}
&(\partial_t^\nu-\Delta_{\mathbb{S}^2_1})u_{\nu}(x, t; x_0, t_0)\\
&=\sum_{l = 0}^{\infty} \sum_{m=-l}^{+l}\, b_{lm}(x_0)[Y_{lm}(x)\partial^{\nu}_{t}\, E_\nu(-\mu_l(t-t_0)^\nu) -E_\nu(-\mu_l(t-t_0)^\nu)\Delta_{\mathbb{S}^2_1}Y_{lm}(x)]=0,
\end{split}
\end{equation*}
so that the PDE in (\ref{pdeFracSphere}) is satisfied.
Thus, we conclude that $u$ defined by (\ref{formal-sol-L-1}) is a
classical (strong) solution to (\ref{pdeFracSphere}).


{\bf Step 5.} We show that $\mathfrak{B}^x(\mathfrak{L}(t))$ is a stochastic solution to \eqref{pdeFracSphere}.
 The solution to $\partial_t u_1 = \triangle_{\mathbb{S}^2_1} u_1$ subject to the initial condition $u_0$ is written as $u_1(x,t;x_0, t_0) = \mathbb{E} u_0\left( \mathfrak{B}^{x}(t-t_0)-x_0\right), \quad x\in \mathbb{S}^2_1,\; t>t_0\geq 0$ or, in terms of the convolution semigroup, as
\begin{equation*}
u_1(x,t;x_0, t_0) = P_{t-t_0}\, u_0(x - x_0)
\end{equation*}
where  the strongly continuous semigroup $P_t$ on $L^2(\mathbb{S}^2_1)$  is satisfies
\begin{equation*}
\lim_{t\downarrow 0} \frac{P_t u_1 - u_1}{t} = \triangle_{\mathbb{S}^2_1}u_1
\end{equation*}
Since  $\triangle_{\mathbb{S}^2_1}$ is self-adjoint operator by Remark 3.34 in \cite{MarPeccBook}. We define the operator
\begin{equation*}
P^\nu_t f(x)= \int_0^\infty   \, l_\nu(s, t)\, P_sf(x) ds.
\end{equation*}


 Finally, we obtain the stochastic representation of the solution.
Since $\{Y_{lm}\}$ forms a complete orthonormal basis for
$L^2(\mathbb{S}^2_1),$ the
 semigroup
$${\E}[u_0(\mathfrak{B}^x(t-t_0)-x_0)]=P_{t-t_0}\, u_0(x - x_0)=\sum_{l=0}^\infty \sum_{m=-l}^l e^{-\mu_l t}b_{lm}(x_0)Y_{lm}(x)$$
is the unique solution to \eqref{brownian-pde}  where
$$
\int_{\mathbb{S}^2_1} P_{t-t_0}\, u_0(x - x_0)Y^*_{lm}(x)\lambda(dx)=e^{-\mu_l t}b_{lm}(x_0)
$$
we can write
by using Fubini  Theorem together with   equations \eqref{lapL} and \eqref{formal-sol-L-1}
 to get
\begin{eqnarray}
u_\nu(x,t;x_0,t_0)&=&\sum_{l=0}^\infty\sum_{m=-l}^lb_{lm}(x_0)E_\nu(-\mu_l(t-t_0)^\nu) Y_{lm}(x)=\sum_{l=0}^\infty\sum_{m=-l}^l a_{lm}(t;x_0,t_0) Y_{lm}(x)\nonumber\\
&=&\sum_{l=0}^\infty\sum_{m=-l}^l b_{lm}(x_0)Y_{lm}(x)\int_0^{\infty}e^{-y\mu_{l}}l_\nu(t,y)dy\nonumber\\
&=&\int_0^{\infty}\left[\sum_{l=0}^\infty\sum_{m=-l}^lb_{lm}(x_0)Y_{lm}(x)e^{-y\mu_l}\right]l_\nu(t,y)dy\nonumber\\
&=&\int_{0}^{\infty}P_{y-t_0}\, u_0(x - x_0)l_\nu(t,y)dy=P^\nu_{t-t_0}u_0(x-x_0)\nonumber\\
&=&\int_{0}^{\infty}{\E}[u_0(\mathfrak{B}^x(y)-x_0)]l_\nu(t,y)dy\nonumber\\
&=&{\E}[u_0(\mathfrak{B}^x(\mathfrak{L}^\nu_{t-t_0}))].\nonumber
\end{eqnarray}

 Uniqueness follows by considering two solutions $u_1,u_2$ with the same initial data, and showing that $u_1-u_2\equiv 0$.

\end{proof}
\begin{os}\label{TRD-density-remark}
We notice that if the initial condition in \eqref{pdeFracSphere} is the Dirac delta function
$$u_0(x-x_0)=\delta(x-x_0)$$
then, from the completeness relationship for spherical harmonics
\[ \delta(x - x_0) = \lim_{n \to \infty} \delta_n(x-x_0) = \lim_{n \to \infty} \sum_{l=0}^{n}\sum_{m=-l}^{+l} Y_{lm}^{*}(x_0)\, Y_{lm}(x) \]
we can write that
\[ a_{lm}(t; x_0, t_0) = Y_{lm}^{*}(x_0) \, E_{\nu}(-\mu_l\, (t - t_0)^\nu).  \]
As we can check, this implies that $R_l \equiv 1$ for all $l \geq 0$ with
\begin{equation*}
A_l(t, t_0) = \frac{2l+1}{4\pi}, \quad \textrm{for all }\; t, t_0>0
\end{equation*}
and also that
$$\delta_n(0) = \sum_{l=0}^{n} \frac{2l+1}{4\pi}.$$
\end{os}

\begin{os}
We check that the transition density \eqref{lawFracRot} integrates to unity. We recall that $Y_{00}(x) = 1/\sqrt{4\pi}$ for all $x \in \mathbb{S}^2_1$. Furthermore, from the integral \eqref{appendixInt2Y}, we have that
\begin{equation*}
\int_{\mathbb{S}^2_1} Y_{lm}(x)\lambda(dx) = \sqrt{4\pi} \int_{\mathbb{S}^2_1} Y_{lm}(x)Y_{00}(x)\lambda(dx) = \sqrt{4\pi} \delta_{l}^0 \delta_m^0.
\end{equation*}
From this and the fact that $E_{\nu}(0)=1$, we obtain
$$\int_{\mathbb{S}^2_1} u_{\nu}(x, t; x_0, t_0) \lambda(dx) = \int_{\mathbb{S}^2_1} u_{\nu}(x, t; x_0, t_0) \lambda(dx_0) = R_0$$
which, under the hypothesis of Theorem \ref{TheoremMain}, leads to the claim.
\end{os}

\begin{os}
For $x_0=x_N$, that is the North Pole, we have that (see Section 5.13.2 of \cite{VMK08})
\begin{equation*}
Y_{lm}(x_0) = 0 \quad \textrm{for } m \neq 0 \textrm{ and } Y_{l0}(x_0) = \sqrt{\frac{2l+1}{4\pi}}
\end{equation*}
whereas (\cite{VMK08}), for $x = (\vartheta, \varphi) \in \mathbb{S}^2_1$,
\begin{equation}
Y_{l0}(\vartheta, \varphi) = \sqrt{\frac{2l+1}{4\pi}} P_l(\cos \vartheta). \label{Y10allx}
\end{equation}
The distribution \eqref{lawFracRot} becomes
\begin{equation}
u_\nu(\vartheta, \varphi, t; x_N, t_0) = \sum_{l \geq 0} \frac{2l+1}{4\pi} R_l\, E_\nu(-\mu_l (t - t_0)^\nu)\, P_l(\cos \vartheta)
\end{equation}
for $0 \leq \vartheta \leq \pi$ and $t>t_0 \geq 0$.
\end{os}

\begin{os}
For $\nu \to 1$ we have that (see \cite{Btoi96}) $\mathfrak{L}^{1}_t \stackrel{a.s.}{\to} t$ and therefore the inverse process $\mathfrak{L}^{1}_t$, $t>0$, becomes the elementary subordinator $t ,\, t\geq 0$.  On the other hand the fractional equation of the Cauchy problem \eqref{pdeFracSphere}, for $\nu=1$, takes the following form
\begin{equation*}
\partial_t u_1 = \triangle_{\mathbb{S}_{1}^2} u_1\\
\end{equation*}
with solution, from the classical Sturm-Liouville theory, given by
\begin{equation}
u_1(x, t; x_0, t_0) = \sum_{l \geq 0} \sum_{m =-l}^{+l} R_l\, \exp\left( -\mu_l (t-t_0)\right)\, Y_{lm}(x) \,Y_{lm}^{*}(x_0)\label{lawBrot}
\end{equation}
which is in accord with the property of the Mittag-Leffler function $E_{1}(-z)= e^{-z}$. Formula \eqref{lawBrot} can be also written as
\begin{equation*}
u_1(x, t; x_0, t_0) = \mathbb{E}u_0( \mathfrak{B}^x(t-t_0)-x_0), \quad x,x_0 \in \mathbb{S}^2_1,\; t>t_0
\end{equation*}
where $\mathfrak{B}^x(t)$, $t>0$ is the rotational Brownian motion or Brownian motion on $\mathbb{S}^2_1$ started at $x$ at $t=t_0$.
\end{os}

\begin{os}
For $\nu \neq 1$ and $u_0 \neq \delta$ the following holds true
\begin{align}
u_\nu(x_2,t_2; x_0, t_0) \neq \int_{\mathbb{S}^2_1}  u_\nu(x_2,t_2; x_1, t_1)\, u_\nu(x_1,t_1; x_0, t_0)\, \lambda(dx_1) \label{non-Markov}
\end{align}
where $t_0 < t_1 < t_2$ and $x_0, x_1,x_2 \in \mathbb{S}^2_1$. Indeed, from \eqref{reproducK} we have that
\begin{equation*}
\int_{\mathbb{S}^2_1} P_{l_1}(\langle x_2, x_1 \rangle)\, P_{l_2}(\langle x_1, x_0 \rangle )\, \lambda(dx_1) = \delta_{l_1}^{l_2} \frac{4\pi}{2l_1 + 1} P_{l_1}(\langle x_2, x_0 \rangle)
\end{equation*}
and therefore
\begin{align*}
& \int_{\mathbb{S}^2_1}  u_\nu(x_2,t_2; x_1, t_1)\, u_\nu(x_1,t_1; x_0, t_0)\, \lambda(dx_1)\\
= & \sum_{l \geq 0} \frac{2l+1}{4\pi} R_l^2 \, E_{\nu}(-\mu_l (t_2-t_1)^\nu)\, E_{\nu}(-\mu_l (t_1-t_0)^\nu)\, P_l(\langle x_2, x_0 \rangle ).
\end{align*}
which coincides with $u_\nu(x_2,t_2; x_0, t_0)$ only if $\nu=1$ and $u_0=\delta$. Indeed, for $\nu=1$,
\begin{equation*}
E_{\nu}(-\mu_l (t_2-t_1)^\nu)\, E_{\nu}(-\mu_l (t_1-t_0)^\nu) = E_{\nu}(-\mu_l (t_2-t_0)^\nu)
\end{equation*}
and, for $u_0=\delta$, $R_l \equiv 1$ for all $l \geq 0$.
\end{os}

\begin{os}
We observe that
\begin{equation*}
\mathbb{P}\{ d(x, \mathfrak{B}^x_{\nu}(t)) \leq \rho \} = \mathbb{P}\left\lbrace  \tau^{t_0}_{\mathcal{B}(x, \rho)}(\mathfrak{B}^x_{\nu}) > t >t_0 \right\rbrace
\end{equation*}
where
$$\tau^{t_0}_{\mathcal{B}(x, \rho)}(\mathfrak{B}^x_{\nu}) = \inf \{ s\geq t_0\,: \, \mathfrak{B}^x_{\nu}(s) \notin \mathcal{B}(x, \rho)  \}$$
and
$$\mathcal{B}(x, \rho) = \{ y \in \mathbb{S}^2_r \,:\, d(x,y) \leq \rho \}, \quad \rho \in (0, \pi).$$
From the fact that $\mathfrak{B}^x_{\nu} \stackrel{d}{=} \mathfrak{B}^{gx}_{\nu}$ for all $g \in SO(3)$, we can also write
\begin{equation*}
\mathbb{P}\left\lbrace \tau^{t_0}_{\mathcal{B}(x, \rho)}(\mathfrak{B}^x_{\nu}) > t\right\rbrace = \mathbb{P}\left\lbrace \tau^{0}_{\mathcal{B}(0, \rho)}(\mathfrak{B}^0_{\nu}) > t-t_0 \right\rbrace
\end{equation*}
or equivalently that
\begin{equation*}
\tau^{t_0}_{\mathcal{B}(x, \rho)}(\mathfrak{B}^x_{\nu}) \stackrel{d}{=} \tau^{0}_{\mathcal{B}(0, \rho)}(\mathfrak{B}^0_{\nu}), \quad x \in \mathbb{S}^2_1.
\end{equation*}
The above distribution (see for example \cite{CamOrs12}) solves the Dirichlet problem
\begin{equation}
\left\lbrace
\begin{array}{ll}
\triangle_{\mathbb{S}_{r}^2} \, q(\vartheta, \varphi; \vartheta_0, \varphi_0) = 0, & 0 < \vartheta < \vartheta_0 < \pi, \; \varphi, \varphi_0 \in [0, 2\pi)\\
q(\vartheta_0, \varphi; \vartheta_0, \varphi_0 ) = \delta(\varphi - \varphi_0),
\end{array}
\right .
\end{equation}
and is written as
\begin{equation}
q(\vartheta, \varphi; \vartheta_0, \varphi_0) = \frac{1}{2\pi} \frac{\cos \vartheta_0 - \cos \vartheta}{1 - \cos \vartheta\, \cos \vartheta_0 - \sin \vartheta \, \sin \vartheta_0 \, \cos (\varphi - \varphi_0)}.
\end{equation}
The solution, say $\bar{u}_\nu$, to the problem \eqref{pdeFracSphere} on the bounded domain
\begin{equation*}
 \mathcal{D} = \left\lbrace (\vartheta, \varphi)\, : \, 0 < \vartheta < \vartheta_0 < \pi, \; \varphi, \varphi_0 \in [0, 2\pi),\; (\vartheta_0, \varphi_0) \in \mathbb{S}^2_1 \right\rbrace
\end{equation*}
can be written as
\begin{equation}
\bar{u}_\nu(x, x_0) = \mathbb{E} u_0\left(\mathfrak{B}^{x}(\mathfrak{L}^\nu_t)\right) \, 1_{\left(\tau^{0}_{\mathcal{D}}(\mathfrak{B}^0_{\nu}) > t-t_0\right)}, \quad x \in \mathcal{D},\; t>t_0?
\end{equation}
See \cite{BMN09ann} for more on fractional Cauchy problems in bounded domains.
\end{os}

\section{Random fields indexed by  spherical non-homogeneous surfaces}
\label{SecRFNHS}
\subsection{Random fields indexed by  the sphere}
In this section we will consider a $n$-weakly isotropic random field $\{ T(x);\, x \in \mathbb{S}^2_1 \}$ on the sphere $\mathbb{S}^2_1$ for which
\begin{equation*} \mathbb{E}T(gx) = 0, \quad \textrm{and} \quad T(gx) \stackrel{d}{=} T(x), \; \textrm{ for all } g \in SO(3)
\end{equation*}
 here  we recall that $\overset{d}{=}$ means equality in distribution of stochastic processes and $SO(3)$ the special group of rotations in $\mathbb{R}^{3}$.

As already pointed out in Section \ref{SecRBM}, the triangular array
$$\{Y_{lm}\,:\,l\geq 0,\, m=-l,\ldots ,+l\}$$
of spherical harmonics  represents a set of eigenfunctions for the Laplacian on the sphere and an orthonormal basis for $L^{2}(\mathbb{S}^{2})$. Hence  we can see that $2$-weakly isotropic random fields admit the spectral representation
\begin{equation}
T(x) = \sum_{l \geq 0} T_l(x) = \sum_{l \geq 0}\sum_{|m| \leq l} a_{lm} Y_{lm}(x)\label{TrepSec1}
\end{equation}
where
\begin{equation}
a_{lm}=\int_{\mathbb{S}^2_1} T(x) Y^{*}_{lm}(x) \lambda(dx) \label{almS1}
\end{equation}
 is a set of random spherical harmonics coefficients that
are Fourier random coefficients and $\lambda$ is the Lebesgue measure such that $\lambda(\mathbb{S}^2_1)=4\pi$. In formula \eqref{almS1} the symbols $Y^*_{lm}$ stands for the complex conjugation of $Y_{lm}$ and convergence in \eqref{TrepSec1} holds in the mean square sense, both with respect to $L^{2}(dP)$ for fixed $x$ and with respect to $L^{2}(dP\otimes \lambda (dx)),$ i.e.
\begin{equation*}
\lim_{L\rightarrow \infty }\mathbb{E}\left(T(x)-\sum_{l=0}^{L}\sum_{m=-l}^{+l}a_{lm}Y_{lm}(x)\right)^{2}= 0
\end{equation*}
and
\begin{equation*}
\lim_{L\rightarrow \infty }\mathbb{E}\left[ \int_{\mathbb{S}^{2}}\left( T(x) - \sum_{l=0}^{L} \sum_{m=-l}^{+l} a_{lm} Y_{lm}(x) \right)^{2} \lambda (dx) \right] =0.
\end{equation*}
  Under isotropy, $\{a_{lm}\}$ are zero-mean valued and uncorrelated over $l$ and $m$ with $\mathbb{E}\left| a_{lm}\right| ^{2}=C_{l}$ where $C_l$ is the angular power spectrum of the random field $T$. We notice that (from the isotropy of $T$) $C_l$ depends only on $l$ as shown in \cite{balMar07, marpec10}. Furthermore, for the harmonic coefficients associated with the frequency $l$, we have that (see formula \eqref{Yconj})
\begin{equation}
a_{lm} = (-1)^m a_{l-m}^* \label{propa1}
\end{equation}
and the following orthogonality property holds
\begin{equation}
\mathbb{E}[a_{l_1m_1} a_{l_2m_2}^*] = \delta_{l_1}^{l_2} \delta_{m_1}^{m_2} C_{l_1}\label{propa2}
\end{equation}
($\delta_a^b$ is the Kronecker symbol \eqref{kronSy}). From \eqref{propa2} we immediately get that $\mathbb{E}|a_{lm}|^2 = C_l$, $l \geq 0$. As in \eqref{SobolevS}, for the angular power spectrum $C_l$ we have that
\begin{equation}
\sum_{l\geq 0} \sum_{|m|\leq l} (2l+1)^{2s} \mathbb{E}|a_{lm}|^2 = \sum_{l\geq 0}(2l+1)^{2s} (2l+1)C_l < \infty \label{SobolevT}
\end{equation}
for some $s \geq 0$. In the light of the previous discussion, for the zero-mean $n$-weakly isotropic random field $T$ we can write the following expression for the covariance function
\begin{equation}
\mathbb{E}[T(x)T(y)] = \sum_{l\geq 0} \frac{2l+1}{4\pi} C_l \, P_l(\langle x, y \rangle), \quad x, y \in \mathbb{S}^2_1 \label{TcovarianceFunction}
\end{equation}
where $\langle x, y \rangle$ is the angle between $x$ and $y$. Indeed, from the representation \eqref{TrepSec1}, we have that
\begin{equation}
\mathbb{E}[T(x)T(y)] = \sum_{l_1 \geq 0} \sum_{l_2 \geq 0} \mathbb{E}[T_{l_1}(x)T_{l_2}(y)]
\end{equation}
where
\begin{equation*}
T_{l_i}(z) = \sum_{m=-l_i}^{+l_i} a_{l_i m} Y_{l_i m }(z), \quad i=1,2
\end{equation*}
and, by taking into account formulae \eqref{propa1} and \eqref{propa2}, we get that
\begin{align*}
\mathbb{E}[T_{l_1}(x)T_{l_2}(y)] = & \sum_{|m_1|\leq l_1} \sum_{|m_2|\leq l_2} \delta_{l_1}^{l_2} \delta_{m_1}^{m_2}C_{l_1}\, Y_{l_1m_1}(x) Y^*_{l_2m_2}(y)\\
= & C_{l_1} \sum_{|m_1|\leq l_1} Y_{l_1m_1}(x) Y^*_{l_2m_2}(y).
\end{align*}
From the addition formula \eqref{additionT} the covariance \eqref{TcovarianceFunction} immediately follows.

 By considering that $P_l(\langle x, x \rangle)=1$, from \eqref{TcovarianceFunction} and \eqref{SobolevT} we also obtain that
\begin{align}
\mathbb{E}[T(x)]^2 = & \sum_{l \geq 0} \frac{2l+1}{4\pi}C_l < \infty. \label{varT}
\end{align}
Under isotropy, the harmonic coefficients $a_{lm}$ are zero-mean and uncorrelated over $l$. Furthermore, if $T$ is Gaussian, then $a_{lm}$ are Gaussian and independent random coefficients (see \cite{balMar07}). We will use throughout also the following fact (see \cite{marpec10})
\begin{equation}
\mathbb{E} [a_{l_1 m_1} \ldots a_{l_n m_n}] = (-1)^{m_n} \sum_{\lambda_1 \ldots \lambda_{n-3}} C_{l_1 m_1 \ldots l_{n-1}m_{n-1}}^{\lambda_1 \ldots \lambda_{n-3}l_n -m_n} \, P_{l_1, \dots , l_n}(\lambda_1, \ldots , \lambda_{n-3})\label{momentan}
\end{equation}
where $C_{l_1 m_1 \ldots l_{n-1}m_{n-1}}^{\lambda_1 \ldots \lambda_{n-2}l_n -m_n}$ is a convolution of Clebsch-Gordan coefficients and $P_{l_1, \dots , l_n}(\cdot)$ is the reduced polyspectrum associated with the isotropic random field $T$. If $n=3$ for instance, then the formula \eqref{momentan} becomes
\begin{equation}
\mathbb{E}[a_{l_1m_1}a_{l_2m_2}a_{l_3m_3}]= (-1)^{m_3} C^{l_3-m_3}_{l_1m_1l_2m_2} P_{l_1l_2l_3}.\label{ThirdOrdMom}
\end{equation}

For more information on random fields on the sphere the reader can consult the book by \cite{MarPeccBook} whereas, we refer to the book by \cite{AdTay07} for random fields and geometry.

\subsection{Random fields indexed by  the non-homogeneous sphere}
In this section we focus on the main object of the work which is the time-changed (we mean composition of processes) random field
\begin{equation}
\mathfrak{T}^{\nu}_t (x)=T(\mathfrak{B}^{x}_{\nu}(t)) = T(x + \mathfrak{B}(\mathfrak{L}^\nu_{t-t_0})), \quad t>t_0 \geq 0,\; x \in \mathbb{S}^2_1, \; \nu \in (0,1] \label{RFQdef}
\end{equation}
where $\mathfrak{B}^{x}_{\nu}(t)$, $t>t_0$ is the time-changed rotational Brownian motion starting from $x$ at time $t_0$ and $\{T(x), \; x \in \mathbb{S}^2_1\}$ is the (Gaussian) random field depicted in the previous section. We assume that $T$ is independent from $\mathfrak{B}^{x}_{\nu}$. Random field with randomly varying parameters have been also studied in the interesting  work \cite{AlloubaNane12}. Based on the representation  \eqref{TrepSec1} of $T$, for the random field
$$\{\mathfrak{T}^{\nu}_t (x):\, x \in \mathbb{S}^2_1,\, t>t_0\geq 0\}, \quad \nu \in (0,1]$$
we write
\begin{equation}
\mathfrak{T}^{\nu}_t (x) = \sum_{l=0}^\infty \sum_{m=-l}^l a_{lm} Y_{lm}(\mathfrak{B}^{x}_{\nu}(t)), \quad t>t_0, \; x \in \mathbb{S}^2_1, \; \nu \in (0,1]
\end{equation}
which must be understood in the $L_2$ sense, that is
\begin{equation*}
\lim_{L \to \infty} \mathbb{E}\Bigg[ \mathbb{E} \bigg\{\Big[T(\mathfrak{B}^{x}_{\nu}(t)) - \sum_{l=0}^{L} \sum_{m=-l}^{+l} a_{lm} Y_{lm}(\mathfrak{B}^{x}_{\nu}(t))  \Big]^2\Bigg| \; \mathfrak{F}_{\mathfrak{B}}\bigg\}  \; \Bigg]  = 0.
\end{equation*}
We recall that the TRD $\mathfrak{B}^{x}_{\nu}(t)$, $t>t_0$ is a measurable map from $(\Omega, \mathfrak{F}_{\mathfrak{B}}, \mathbb{P})$ to $(\mathbb{S}^2_1, \mathcal{B}(\mathbb{S}^2_1), \lambda_{\mathfrak{B}})$.

In the rest of the paper use the double summation notation whenever appropriate $$\sum_{lm}=\sum_{l=0}^\infty\sum_{m=-l}^l.$$ 


\begin{os}
Let us consider the random shift $s_t : x \mapsto \mathfrak{B}^{x}_{\nu}(t)$, such that
\begin{equation}
s_t x = x + \mathfrak{B}^{0}_{\nu}(t-t_0) = \mathfrak{B}^{x}_{\nu}(t) \label{gtAction}
\end{equation}
where $0=x_N$ is the North Pole. The process \eqref{RFQdef} can be therefore regarded as
\begin{equation}
\mathfrak{T}^{\nu}_t (x) = T(s_t x), \quad x \in \mathbb{S}^2_1, \; t>t_0
\end{equation}
where the random field $T$ is indexed by $s_tx$ and therefore under the action of $s_t$. Formula \eqref{gtAction} says that the spherical coordinate $x \in \mathbb{S}^2_1$ is affected by some noise depending on time. As we will see further in the text, for $\nu \to 1$ we get that $\mathfrak{B}^{0}_{\nu} \to \mathfrak{B}^0$ and therefore the noise becomes the standard rotational Brownian motion. For $\nu \in (0,1)$, the random environment in which $T$ is indexed is represented by the TRD $\mathfrak{B}^{x}_{\nu}$ and therefore is a non-homogeneous sphere.
\end{os}

As expected, the properties of \eqref{RFQdef} are strictly related with those of $T$. In particular, we begin our analysis by presenting the following result concerning the higher-order moments of the composition $\mathfrak{T}^{\nu}_t (x)$.
\begin{lem}\label{n-moment-time-changed-field}
Let the previous setting prevail. The following holds
\begin{equation}
\mathbb{E}[\mathfrak{T}^{\nu}_t (gx)]^n = \mathbb{E}[T(x)]^n, \quad  t >t_0 \geq 0,\quad \nu \in (0,1], \quad n \in \mathbb{N}\quad g\in SO(3).
\end{equation}
\label{lemMomenti}
\end{lem}
\begin{proof}
The claimed result follows by observing  that
\begin{align*}
\mathbb{E} [\mathfrak{T}^{\nu}_t (x)]^n = & \mathbb{E} \Bigg[ \mathbb{E} \bigg\{\Big[ \sum_{lm} a_{lm} Y_{lm}(\mathfrak{B}^x_{\nu}(t)) \Big]^n \Big|  \; \mathfrak{F}_{\mathfrak{B}} \;\bigg\} \Bigg] \\
= & \mathbb{E} \Bigg[ \mathbb{E}\bigg\{ \Big[ \sum_{lm} a_{lm} Y_{lm}(y) \Big]^n \Big| y = \mathfrak{B}^x_{\nu}(t)\bigg\}  \Bigg] , \quad \forall x \in \mathbb{S}^2_{1}\\
= & \mathbb{E} \Big[ \mathbb{E} \big\{[ T(y) ]^n \Big| y = \mathfrak{B}^x_{\nu}(t)  \big\}\Big] , \quad \forall x \in \mathbb{S}^2_{1}.
\end{align*}
From the fact that (see \cite{marpec10})
\begin{equation*}
\mathbb{E}[T(y)]^n = \sum_{l_1\ldots l_n} \sqrt{\frac{(2l_1+1)\cdots (2l_n+1)}{(4\pi)^n}} \mathbb{E}[a_{l_10}\cdots a_{l_n 0}]
\end{equation*}
which does not depend on $y \in \mathbb{S}^2_1$ as formula \eqref{momentan} entails, we obtain that
\begin{align*}
\mathbb{E} [\mathfrak{T}^{\nu}_t (x)]^n = & \mathbb{E}[ T(y) ]^n, \quad \forall y \in \mathbb{S}^2_{1}
\end{align*}
which, in turn, does not depend on $x \in \mathbb{S}^2_1$. Furthermore, $\mathfrak{B}^x_{\nu} \stackrel{d}{=} \mathfrak{B}^{gx}_{\nu}$, for all $g \in SO(3)$ and this concludes the proof.
\end{proof}
\begin{prop}
For $x \in \mathbb{S}^2_1$, $t> t_0\geq 0$ and $\nu \in (0,1]$, we have that
\begin{align}
\mathbb{E}Y_{l m}(\mathfrak{B}^{x}_{\nu}(t)) = & E_{\nu}(-\mu_l (t-t_0)^\nu) Y_{lm}(x) \label{EYprop2}
\end{align}
and
\begin{align}
\mathbb{E}Y_{l m}^*(\mathfrak{B}^{x}_{\nu}(t)) = &  E_{\nu}(-\mu_l (t-t_0)^\nu) Y^*_{lm}(x). \label{EYprop1}
\end{align}
\label{propositionEY12}
\end{prop}

\begin{proof}
We recall from Remark \ref{TRD-density-remark} that, for $\nu \in (0,1]$,
\begin{equation}
u_{\nu}(y,t; x,t_0) = \sum_{l =0}^{\infty}  \, E_{\nu}\left( - \mu_l\, (t-t_0)^{\nu} \right) \sum_{m=-l}^{+l} Y_{lm}(y)Y^{*}_{lm}(x) \label{distBanomalousProp}
\end{equation}
$y,x \in \mathbb{S}^2_1$, $t > t_0 \geq 0$, is the probability density of the TRD $\mathfrak{B}^{x}_\nu(t)$, $t>t_0$. From the property \eqref{Yconj} and the orthogonality of spherical harmonics \eqref{Yortho}, we get that
\begin{equation*}
\int_{\mathbb{S}^2_1} Y^*_{\gamma \kappa}(y)\, Y_{lm}(y)\, \lambda(dy) = \delta_{\gamma \kappa}^{lm}
\end{equation*}
and thus
\begin{align}
\mathbb{E}[Y_{lm}(\mathfrak{B}^{x}_{\nu}(t))] = & \int_{\mathbb{S}^2_1} u_\nu(y,t;x,t_0) Y_{lm}(y)\, \lambda(dy)\nonumber \\
= & \sum_{\gamma=0}^\infty\sum_{\kappa=-\gamma} ^{\gamma} \, E_{\nu}(-\mu_\gamma (t-t_0)^{\nu}) Y_{\gamma \kappa}^{*}(x) (-1)^m \int_{\mathbb{S}^2_1} Y_{\gamma \kappa}(y) Y_{l-m}^{*}(y) \lambda(dy) \nonumber \\
= &  \sum_{\gamma=0}^\infty\sum_{\kappa=-\gamma} ^{\gamma} \, E_{\nu}(-\mu_\gamma(t-t_0)^{\nu})  (-1)^m Y_{\gamma \kappa}^{*}(x) \delta_{\gamma}^l \delta_{\kappa}^{-m}\nonumber \\
= & \, E_{\nu}(-\mu_l\,(t-t_0)^{\nu}) (-1)^m Y^*_{l -m}(x)\notag \\
= & [\textrm{by } \eqref{Yconj}] = \, E_{\nu}(-\mu_l\,(t-t_0)^{\nu})  Y_{lm}(x).\label{EYprof}
\end{align}
Formula \eqref{EYprop1} can be obtained from result \eqref{EYprof} by considering the fact that
\begin{equation}
\mathbb{E} [Y_{lm}^{*}(\mathfrak{B}^{x_0}_{\nu}(t))] = \left( \mathbb{E} [Y_{lm}(\mathfrak{B}^{x_0}_{\nu}(t))] \right)^{*} = \, E_{\nu}(-\mu_l \,(t-t_0)^{\nu})  Y_{lm}^{*}(x_0)
\end{equation}
and the proof is completed.
\end{proof}

\begin{prop}
For $x \in \mathbb{S}^2_1$, $t> t_0\geq 0$ and $\nu \in (0,1]$, we have that
\begin{align}
\mathbb{E} \left[ Y_{lm}(\mathfrak{B}^{x}_{\nu}(t))\,Y_{lm}^{*}(\mathfrak{B}^{x}_{\nu}(t)) \right] =  \sum_{\gamma, \kappa} \mathcal{C}_{\gamma, \kappa}^{lm}  \, \, E_{\nu}(- \mu_\gamma (t-t_0)^\nu) Y^*_{\gamma \kappa}(x) \label{respropEYY}
\end{align}
where
\begin{equation*}
\mathcal{C}_{\gamma, \kappa}^{lm} = (-1)^m (2l+1) \sqrt{\frac{(2\gamma +1)}{4 \pi}} \left( \begin{array}{ccc} \gamma & l & l \\ 0 & 0 & 0  \end{array} \right) \left( \begin{array}{ccc} \gamma & l & l \\ \kappa & m& -m  \end{array} \right)
\end{equation*}
is written in terms of Wigner $3j$-symbols (see Appendix \ref{AppA}).
\label{propEYY}
\end{prop}
\begin{proof}
From \eqref{distBanomalousProp} we write
\begin{align*}
& \mathbb{E} \left[ Y_{lm}(\mathfrak{B}^{x}_{\nu}(t))\,Y_{lm}^{*}(\mathfrak{B}^{x}_{\nu}(t)) \right]\\
= & \int_{\mathbb{S}^2_1} u_{\nu}(y,t; x,t_0)\,  Y_{lm}(y)\, Y_{lm}^{*}(y)\, \lambda(dy) \\
= & \sum_{\gamma=0}^\infty\sum_{\kappa=-\gamma} ^{\gamma} \, E_{\nu}(- \mu_\gamma (t-t_0)^\nu) Y^*_{\gamma \kappa}(x) \int_{\mathbb{S}^2_1}  Y_{\gamma \kappa}(y)\,Y_{lm}(y)\, Y_{lm}^{*}(y)\, \lambda(dy)\\
= & (-1)^m \sum_{\gamma=0}^\infty\sum_{\kappa=-\gamma} ^{\gamma} \, E_{\nu}(- \mu_\gamma (t-t_0)^\nu) Y^*_{\gamma \kappa}(x)  \int_{\mathbb{S}^2_1}  Y_{\gamma \kappa}(y)\, Y_{lm}(y)\, Y_{l-m}(y)\, \lambda(dy).
\end{align*}
From the fact that (see formula \eqref{intprodtreY})
\begin{equation*}
\int_{\mathbb{S}^2_1} Y_{\gamma \kappa}(y) \, Y_{lm}(y) \, Y_{l-m}(y)\, \lambda(y) = \Psi_{\gamma}^{l} \left( \begin{array}{ccc} \gamma & l & l \\ 0 & 0 & 0  \end{array} \right) \left( \begin{array}{ccc} \gamma & l & l \\ \kappa & m& -m  \end{array} \right)
\end{equation*}
where
\begin{equation}
\Psi_{\gamma}^{l} = (2l+1) \sqrt{\frac{(2\gamma +1)}{4 \pi}},
\end{equation}
we arrive at formula \eqref{respropEYY}.
\end{proof}

%

\begin{os}
We use an alternative approach in order to show that
\begin{equation*}
Var[\mathfrak{T}^{\nu}_t (x)]=Var[T(x)] = \sum_{l \geq 0} \frac{2l+1}{4\pi}C_l < \infty
\end{equation*}
or equivalently that
\begin{equation*}
\mathbb{E}[T(s_t gx)]^2 = \mathbb{E}[T(x)]^2, \quad \forall\, g \in SO(3), \quad t>t_0 \geq 0
\end{equation*}
as already stated in Lemma \ref{lemMomenti}.
\end{os}
\begin{proof}
 We have that
\begin{align*}
\mathbb{E}[\mathfrak{T}^{\nu}_t (x)]^2 = & \mathbb{E} \Bigg[ \mathbb{E} \bigg\{\Big[ \sum_{lm} a_{lm} Y_{lm}(\mathfrak{B}^{x}_{\nu}(t)) \Big]^{2} \Big| \mathfrak{F}_{\mathfrak{B}} \bigg\}\Bigg]\\
= & \mathbb{E} \Bigg[ \mathbb{E} \bigg\{\Big[ \sum_{l_1m_1} \sum_{l_2m_2} a_{l_1m_1} a_{l_2m_2}\, Y_{l_1m_1}(\mathfrak{B}^{x}_{\nu}(t)) \, Y_{l_2m_2}(\mathfrak{B}^{x}_{\nu}(t)) \Big] \Big| \mathfrak{F}_{\mathfrak{B}} \bigg\} \Bigg]
\end{align*}
where
\begin{equation}
\mathbb{E} [a_{l_1m_1} a_{l_2m_2}] = (-1)^{m_1} \delta_{l_1}^{l_2}\delta_{m_1}^{-m_2}C_{l_1}
\end{equation}
and $(-1)^m Y_{l-m} = Y^{*}_{lm}$. Thus, we obtain
\begin{align*}
\mathbb{E}[\mathfrak{T}^{\nu}_t (x)]^2 = & \mathbb{E} \Bigg[ \sum_{l_1m_1} (-1)^{m_1} C_{l_1}\, Y_{l_1m_1}(\mathfrak{B}^{x}_{\nu}(t)) \, Y_{l_1-m_1}(\mathfrak{B}^{x}_{\nu}(t))  \Bigg]\\
= & \sum_{l_1m_1} C_{l_1}\, \mathbb{E} \left[ Y_{l_1m_1}(\mathfrak{B}^{x}_{\nu}(t)) \, Y_{l_1m_1}^{*}(\mathfrak{B}^{x}_{\nu}(t))  \right]
\end{align*}
where
\begin{align}
\mathbb{E} \left[ Y_{lm}(\mathfrak{B}^{x}_{\nu}(t))\,Y_{lm}^{*}(\mathfrak{B}^{x}_{\nu}(t)) \right] =  \sum_{\gamma \kappa} \mathcal{C}_{\gamma, \kappa}^{lm}  \, \, E_{\nu}(- \mu_\gamma (t-t_0)^\nu) Y^*_{\gamma \kappa}(x)
\end{align}
and
\begin{equation*}
\mathcal{C}_{\gamma, \kappa}^{lm} = (-1)^m (2l+1) \sqrt{\frac{(2\gamma +1)}{4 \pi}} \left( \begin{array}{ccc} \gamma & l & l \\ 0 & 0 & 0  \end{array} \right) \left( \begin{array}{ccc} \gamma & l & l \\ \kappa & m& -m  \end{array} \right)
\end{equation*}
as shown in Proposition \ref{propEYY}. From this, we write
\begin{align*}
& \sum_{m=-l}^{+l} \mathbb{E} \left[ Y_{lm}(\mathfrak{B}^{x}_{\nu}(t))\,Y_{lm}^{*}(\mathfrak{B}^{x}_{\nu}(t)) \right]\\
= &  \sum_{m=-l}^{+l} (-1)^m \sum_{\gamma, \kappa} \Psi_{\gamma}^{l}\, \, E_{\nu}(- \mu_\gamma (t-t_0)^\nu) Y^*_{\gamma \kappa}(x) \left( \begin{array}{ccc} \gamma & l & l \\ 0 & 0 & 0  \end{array} \right) \left( \begin{array}{ccc} \gamma & l & l \\ \kappa & m& -m  \end{array} \right)\\
= & \sum_{\gamma \kappa} \Psi_{\gamma}^{l}\, \, E_{\nu}(- \mu_\gamma (t-t_0)^\nu) Y^*_{\gamma \kappa}(x) \left( \begin{array}{ccc} \gamma & l & l \\ 0 & 0 & 0  \end{array} \right) \sum_m (-1)^m  \left( \begin{array}{ccc} \gamma & l & l \\ \kappa & m& -m  \end{array} \right)
\end{align*}
where
\begin{equation*}
\Psi_{\gamma}^{l} = (2l+1) \sqrt{\frac{(2\gamma +1)}{4 \pi}}.
\end{equation*}
From \eqref{orth3} and the fact that (\cite{VMK08})
\begin{equation*}
\left( \begin{array}{ccc}
l&l&0 \\ 0&0&0
\end{array} \right) = \frac{1}{\sqrt{2l+1}}
\end{equation*}
we get that
\begin{equation*}
\left( \begin{array}{ccc} \gamma & l & l \\ 0 & 0 & 0  \end{array} \right) \sum_{m=-l}^{+l} (-1)^m  \left( \begin{array}{ccc} \gamma & l & l \\ \kappa & m& -m  \end{array} \right) = \delta_{\gamma}^{0} \delta_{\kappa}^{0}
\end{equation*}
and, by taking into account the fact that (see for example \cite{VMK08})
$$Y_{00}(\vartheta, \phi) = \frac{1}{2\sqrt{\pi}}$$
and $E_\nu(0)=1$, $R_0=1$, we obtain that
\begin{align*}
\mathbb{E}[\mathfrak{T}_t^{\nu}(x)]^2 = & \sum_{lm} C_l\, \mathbb{E} \left[ Y_{lm}(\mathfrak{B}^{x}_{\nu}(t))\,Y_{lm}^{*}(\mathfrak{B}^{x}_{\nu}(t)) \right] \\
= & \sum_{l} C_l\, \sum_{\gamma \kappa} \Psi_{\gamma}^{l}\, \, E_{\nu}(- \mu_\gamma (t-t_0)^\nu) Y^*_{\gamma \kappa}(x)  \delta_{\gamma}^0\, \delta_{\kappa}^{0}\\
= & Y_{00}^{*}(x) \sum_l C_l\, \frac{2l+1}{\sqrt{4\pi}} = \mathbb{E}[T(x)]^2, \quad \forall x \in \mathbb{S}^2_{1}
\end{align*}
which is the claimed result.

\end{proof}
In order to study the time- and space-covariance of the random field
\begin{equation*}
\{\mathfrak{T}^{\nu}_t (x)=T(\mathfrak{B}^{x}_{\nu}(t)) ;\, x \in \mathbb{S}^2_1,\, t>t_0 \geq 0 \}, \quad \nu \in (0,1].
\end{equation*}
We assume that
for given starting points $x, y \in \mathbb{S}^2_1$ at $t_0 \geq 0$, the processes $\mathfrak{B}^{x}_{\nu}(t_1)$ and $\mathfrak{B}^{y}_{\nu}(t_2)$ are independent if and only if $x \neq y$. If $x=y$, then $\mathfrak{B}^{x}_{\nu}(t_1)$ and $\mathfrak{B}^{y}_{\nu}(t_2)$ are two dependent copies of a time-changed rotational Brownian motion starting from $x=y$ at $t_0$.

From the representation
\begin{equation}
\mathfrak{T}^{\nu}_t (x) = T(\mathfrak{B}^x_{\nu}(t)) = \sum_{l = 0}^{\infty} \sum_{m=-l}^{+l} a_{lm} Y^{*}_{lm}(\mathfrak{B}^x_{\nu}(t)), \quad t>t_0
\end{equation}
we can write
\begin{align}
\mathbb{E}[ \mathfrak{T}^{\nu}_{t_1}(x)\, \mathfrak{T}^{\nu}_{t_2} (y) ] =  & \mathbb{E}\Bigg[ \mathbb{E} \bigg\{\Big[ T(\mathfrak{B}^x_{\nu}(t_1))\, T(\mathfrak{B}^y_{\nu}(t_2))  \Big] \Big| \; \mathfrak{F}_{\mathfrak{B}} \otimes \mathfrak{F}_{\mathfrak{B}}\bigg\} \; \Bigg]\nonumber \\
= & \mathbb{E}\Bigg[ \sum_{lm} C_l\, Y_{lm}(\mathfrak{B}^x_{\nu}(t_1))\, Y^{*}_{lm}(\mathfrak{B}^y_{\nu}(t_2)) \Bigg|  \; \mathfrak{F}_{\mathfrak{B}} \otimes \mathfrak{F}_{\mathfrak{B}} \; \Bigg]\nonumber \\
= & \sum_{lm} C_l\, \mathbb{E}\left[ Y_{lm}(\mathfrak{B}^x_{\nu}(t_1))\, Y^{*}_{lm}(\mathfrak{B}^y_{\nu}(t_2)) \right] \label{ProvCor}
\end{align}
for some $t_1, t_2 \geq 0$ and $x,y \in \mathbb{S}^2_1$. Next we give

\begin{proof}[{\bf Proof of Theorem \ref{second-main-thm}}]
We first consider the starting point $y=x$ at time $t_0 \geq 0$. \\
$i)$ Due to the fact that
\begin{align*}
\P\{ \mathfrak{B}^{x}_{\nu}(t_1) \in dy\, , \mathfrak{B}^x_{\nu}(t_0) \in dx \} = & \P\{ \mathfrak{B}^x_{\nu}(t_0) \in dx \}\,\P\{ \mathfrak{B}^{x}_{\nu}(t_1) \in dy\, | \mathfrak{B}^x_{\nu}(t_0) \in dx \} \\
= & \frac{1}{\lambda(\mathbb{S}^2_1)} \P\{ \mathfrak{B}^{x}_{\nu}(t_1) \in dy\, | \mathfrak{B}^x_{\nu}(t_0) \in dx \},
\end{align*}
for $0 \leq t_0 < t_1 < \infty$, we can write
\begin{align}
& \mathbb{E} \left[ Y_{lm}(\mathfrak{B}^{x}_{\nu}(t_1))\,Y_{lm}^{*}(\mathfrak{B}^{x}_{\nu}(t_0)) \right] \label{covst}\\
=&Y_{lm}^{*}(x)\mathbb{E} \left[ Y_{lm}(\mathfrak{B}^{x}_{\nu}(t_1))\, \right] \\
=& Y_{lm}^{*}(x)Y_{lm}(x)E_\nu (-\mu_l (t_1 - t_0)^\nu) \ \ \mathrm{by}\ equation \ \eqref{EYprop2}.
\end{align}
 Thus, formula \eqref{ProvCor} takes the form
\begin{align*}
\sum_{lm} C_l\, \mathbb{E} \left[ Y_{lm}(\mathfrak{B}^{x}_{\nu}(t_1))\,Y_{lm}^{*}(\mathfrak{B}^{x}_{\nu}(t_0)) \right] = &\sum_{lm} C_l\, Y_{lm}^{*}(x)Y_{lm}(x)E_\nu (-\mu_l (t_1 - t_0)^\nu) \, E_\nu (-\mu_l (t_1 - t_0)^\nu).\\
=&\sum_{lm}\frac{2l+1}{4\pi} C_l\, E_\nu (-\mu_l (t_1 - t_0)^\nu).
\end{align*}
We also have that
\begin{equation*}
\mathbb{E} \left[ Y_{lm}(\mathfrak{B}^{x}_{\nu}(t_1))\,Y_{lm}^{*}(\mathfrak{B}^{x}_{\nu}(t_2)) \right]\\
= \mathbb{E} \left[ Y_{lm}(\mathfrak{B}^{gx}_{\nu}(t_1))\,Y_{lm}^{*}(\mathfrak{B}^{gx}_{\nu}(t_2)) \right], \quad \forall\, g\in SO(3)
\end{equation*}
and this prove the first statement. \\

We consider the case in which $x\neq y$:\\

$ii)$ From the assumption that the processes $\mathfrak{B}^{x}_{\nu}(t_1)$ and $\mathfrak{B}^{y}_{\nu}(t_2)$ are independent if and only if $x \neq y$ we obtain that
\begin{equation}
\mathbb{E} [Y_{lm}(\mathfrak{B}^x_{\nu}(t_1))\, Y^{*}_{lm}(\mathfrak{B}^y_{\nu}(t_2)) ]= \mathbb{E} [Y_{lm}(\mathfrak{B}^x_{\nu}(t_1))] \cdot \mathbb{E}[Y^{*}_{lm}(\mathfrak{B}^y_{\nu}(t_2))], \quad \textrm{iff } x \neq y \label{Sfromth}
\end{equation}
where, from the formulae \eqref{EYprop2} and \eqref{EYprop1} of the Proposition \ref{propositionEY12}, we have that
\begin{equation}
\mathbb{E}[Y_{lm}(\mathfrak{B}^{x}_{\nu}(t_1))]  =  \, E_{\nu}(-\mu_l\,(t_1-t_0)^{\nu})  Y_{lm}(x) \label{meanY1}
\end{equation}
and
\begin{equation}
\mathbb{E} [Y_{lm}^{*}(\mathfrak{B}^{y}_{\nu}(t_2))] = \, E_{\nu}(-\mu_l \,(t_2-t_0)^{\nu})  Y_{lm}^{*}(y). \label{meanY2}
\end{equation}
For the sake of simplicity, we write
\begin{equation*}
 \mathcal{E}_l(t_0, t_1, t_2) = E_{\nu}(-\mu_l \,(t_1-t_0)^{\nu})E_{\nu}(-\mu_l \,(t_2-t_0)^{\nu}).
\end{equation*}
From \eqref{meanY1} and \eqref{meanY2}, the formula \eqref{ProvCor} takes the form
\begin{equation*}
\sum_{lm} C_l\, \mathbb{E}\left[ Y_{lm}(\mathfrak{B}^x_{\nu}(t_1))\, Y^{*}_{lm}(\mathfrak{B}^y_{\nu}(t_2)) \right] = \sum_{l \geq 0} C_l\,  \, \mathcal{E}_l(t_0, t_1, t_2) \sum_{|m| \leq l} Y_{lm}(x) \, Y_{lm}^*(y).
\end{equation*}
By taking into account the addition formula \eqref{additionT} we get that
\begin{equation*}
\sum_{lm} C_l\, \mathbb{E}\left[ Y_{lm}(\mathfrak{B}^x_{\nu}(t_1))\, Y^{*}_{lm}(\mathfrak{B}^y_{\nu}(t_2)) \right] = \sum_{l \geq 0} C_l\,  \, \mathcal{E}_l(t_0, t_1, t_2) \frac{2l+1}{4\pi} P_l(\langle x, y \rangle)
\end{equation*}
and the claimed result immediately follows.
\end{proof}

\begin{os}
\label{RemarkShortLongDep}
 From  \eqref{TrepSec1} we arrive at the representation
$$\mathfrak{T}^\nu_t(x) = \sum_{l \geq 0} \mathfrak{T}^\nu_{l,t}(x)$$
where the projection of $\mathfrak{T}^\nu_t$ on $\mathcal{H}_l$, space generated by $\{Y_{lm}: |m\leq l|\}$,  written as
$$\mathfrak{T}^\nu_{l,t}(x) = \sum_{|m| \leq l} a_{lm}Y_{lm}(\mathfrak{B}^x_\nu(t)), \quad t>0,\; l \geq 0$$
is a random field representing the $l$th frequency component of $\mathfrak{T}^\nu_t$. The asymptotic properties of $\mathfrak{T}^\nu_{l,t}$ (as $l \to \infty$) turn out to be very important in the analysis of the CMB radiation (see \cite{marin06, marin08}). In particular, we have that
\begin{align*}
\mathbb{E}[\mathfrak{T}^\nu_{l_1,t_1}(x) \mathfrak{T}^\nu_{l_2,t_2}(y)] = &  \sum_{|m_1|\leq l_1} \sum_{|m_2| \leq l_2} \mathbb{E}[a_{l_1m_1}a^*_{l_2m_2}] \, \mathbb{E}[Y_{l_1m_1}(\mathfrak{B}^x_\nu(t_1))\, Y^*_{l_2m_2}(\mathfrak{B}^y_\nu(t_2))].
\end{align*}
From \eqref{propa2} we get that
\begin{align*}
\mathbb{E}[\mathfrak{T}^\nu_{l_1,t_1}(x) \mathfrak{T}^\nu_{l_2,t_2}(y)] = & \delta_{l_1}^l\delta_{l_2}^l \sum_{|m|\leq l}  C_{l} \, \mathbb{E}[Y_{lm}(\mathfrak{B}^x_\nu(t_1))\, Y^*_{lm}(\mathfrak{B}^y_\nu(t_2))]\\
= &  \delta_{l_1}^l\delta_{l_2}^l \sum_{|m|\leq l}  C_l \, \mathbb{E}[Y_{lm}(\mathfrak{B}^x_\nu(t_1))]\, \mathbb{E}[Y^*_{lm}(\mathfrak{B}^y_\nu(t_2))]
\end{align*}
where in the last step we used the fact that processes $\mathfrak{B}^{x}_{\nu}(t_1)$ and $\mathfrak{B}^{y}_{\nu}(t_2)$ are independent if and only if $x \neq y$. From Proposition \ref{propositionEY12} we get
\begin{align*}
\mathbb{E}[\mathfrak{T}^\nu_{l_1,t_1}(x) \mathfrak{T}^\nu_{l_2,t_2}(y)] = &  \delta_{l_1}^l\delta_{l_2}^l \sum_{|m|\leq l}  C_l \, \,  E_\nu(-\mu_l (t_1-t_0)^\nu)\, E_\nu(-\mu_l (t_2-t_0)^\nu)\, Y_{lm}(x)Y^*_{lm}(y)
\end{align*}
and, from \eqref{symmetry-equality} we arrive at
\begin{align}
\mathbb{E}[\mathfrak{T}^\nu_{l_1,t_1}(x) \mathfrak{T}^\nu_{l_2,t_2}(y)] = & \delta_{l_1}^l\delta_{l_2}^l \frac{2l+1}{4\pi} C_l \, \, E_\nu(-\mu_l (t_1-t_0)^\nu)\, E_\nu(-\mu_l (t_2-t_0)^\nu)\, P_l(\langle x, y \rangle) \label{cov}
\end{align}
which is the (space/time) covariance of the $l$th frequency component of $\mathfrak{T}^\nu_t$. Formula \eqref{genCor2} can be therefore rewritten as follows
\begin{align}
\mathbb{E}[\mathfrak{T}^{\nu}_{t_1} (x)\, \mathfrak{T}^{\nu}_{t_2}(y) ] = & \sum_{l \geq 0} \mathbb{E}[\mathfrak{T}^\nu_{l,t_1}(x) \mathfrak{T}^\nu_{l,t_2}(y)].  \label{ShortLongRemark}
\end{align}
We are now interested in studying the high-frequency behavior of such covariance. Usually the angular power spectrum of $T$ is assumed to be as $C_l = l^{-\alpha} g(l)$ where $\alpha > 2$ (to ensure summability) and $g(\cdot)$ is a smooth function which converges to a constant as the frequency $l\to \infty$ (for instance $g(l)=G_1(l)/G_2(l)$ and $G_1$, $G_2$ are polynomials of the same order).  The important fact is that $E_\nu$ has different behaviour depending on $\nu$ as we know. Let us take $t_0=0$ for the sake of simplicity. The symbols $\delta_{l_1}^l \delta_{l_2}^l$ in \eqref{cov} say that the $l$th components are uncorrelated for different frequencies $l_1, l_2$ whereas, for large $l$ and  every spherical distance $\langle x, y \rangle$
\begin{align*}
\mathbb{E}[\mathfrak{T}^\nu_{l,t_1+h}(x) \mathfrak{T}^\nu_{l,t_1}(y)] \sim { Ct_1^{-\nu} l^{-\theta_1^*}\, h^{-\nu}}, \; \nu \in (0,1) {\quad if \; t_0=0}
\end{align*}
where $\theta_1^*=\alpha+3 >2$
and, for $\nu=1$,
\begin{align*}
\mathbb{E}[\mathfrak{T}^\nu_{l,t_1+h}(x) \mathfrak{T}^\nu_{l,t_1}(y)] \sim  D l^{-\theta_2^*}\, e^{-h l(l+1)}
\end{align*}
where $\theta_2^*=\alpha-1 >1$ and  $C,D>0$ are some constant that do not depend on $l$ or $h$. This means that series \eqref{ShortLongRemark} converges with rate of convergence depending on $\nu$ and therefore, the $l$th frequency component (random filed) $\mathfrak{T}^\nu_{l,t}$ has different covariance structure as $l \to \infty$. Also, for $\nu \in (0,1)$, the time covariance decays more slowly than an exponential decay (for $\nu=1$).
\end{os}

We next give the  proof of Theorem \ref{covariance-nu-1}

\begin{proof}[{\bf Proof of Theorem \ref{covariance-nu-1}}]
For $\nu=1$, we have a Markovian diffusion as the Chapman-Kolmogorov relation \eqref{non-Markov} entails.   Due to the Markov property we have that
\begin{align*}
\P\{\mathfrak{B}^{x}(t_2) \in dz_2 ,\, \mathfrak{B}^{x}(t_1) \in dz_1 \} = & \P\{\mathfrak{B}^{x}(t_2) \in dz_2 | \, \mathfrak{B}^{x}(t_1) \in dz_1 \}\\
& \times \P\{ \mathfrak{B}^{x}(t_1) \in dz_1| \mathfrak{B}^{x}(t_0) \in dx \} \\
& \times \P\{ \mathfrak{B}^x(t_0) \in dx \}
\end{align*}
for $t_2\geq t_1>t_0\geq 0$ where, for all $x \in \mathbb{S}^2_1$,
\begin{equation*}
\P\{ \mathfrak{B}^{x}(t_0) \in dx \}  = 1/ \lambda(\mathbb{S}^2_1).
\end{equation*}
Thus, we obtain that
\begin{align*}
& \mathbb{E} \left[ Y_{lm}(\mathfrak{B}^{x}(t_1))\,Y_{lm}^{*}(\mathfrak{B}^{x}(t_2)) \right]\\
= & \, \exp(-\mu_{l} (t_2-t_1)) \int_{\mathbb{S}^2} Y_{lm}(z_1)\, Y_{lm}^{*}(z_1)\, u_1(z_1,t_1; x, t_0)\, \lambda(dz_1)\\
= & \,\exp(-\mu_{l} (t_2-t_1)) \sum_{\gamma_2 \kappa_2} \,\exp(-\mu_{\gamma_2} (t_1-t_0)) \\
& \times \int_{\mathbb{S}^2} Y_{lm}(z_1)\, Y_{lm}^{*}(z_1)\, Y_{\gamma_2 \kappa_2}(z_1)\, Y_{\gamma_2 \kappa_2}^{*}(x)\, \lambda(dz_1)\\
= &  \,\exp(-\mu_{l} (t_2-t_1)) \sum_{\gamma_2 \kappa_2} \, \exp(-\mu_{\gamma_2} (t_1-t_0)) \, Y_{\gamma_2 \kappa_2}^{*}(x)\\
& \times \sqrt{\frac{(2l+1)^2 (2\gamma_2 +1)}{4\pi}} (-1)^{m} \left( \begin{array}{ccc} l&l&\gamma_2 \\ 0&0&0 \end{array} \right)  \left( \begin{array}{ccc} l&l&\gamma_2 \\ m&-m&\kappa_2 \end{array} \right)
\end{align*}
where we have used the formula \eqref{intprodtreY} and the fact that $Y_{lm}^{*}(z_1) = (-1)^m Y_{l-m}(z_1)$. We write
\begin{align*}
& \mathcal{C}^{lm}_{\gamma_2\kappa_2} =   (-1)^{m} \,(2l+1) \sqrt{\frac{2\gamma_2 +1}{4\pi}}  \left( \begin{array}{ccc} l&l&\gamma_2 \\ 0&0&0 \end{array} \right) \left( \begin{array}{ccc} l&l&\gamma_2 \\ m&-m&\kappa_2 \end{array} \right)
\end{align*}
and therefore we get that
\begin{equation*}
\mathbb{E} \left[ Y_{lm}(\mathfrak{B}^{x}(t_1))\,Y_{lm}^{*}(\mathfrak{B}^{x}(t_2)) \right] =  \, e^{-\mu_{l} (t_2-t_1)} \sum_{\gamma_2 \kappa_2} \mathcal{C}^{lm}_{\gamma_2\kappa_2}\, \,e^{-\mu_{\gamma_2} (t_1-t_0)} \, Y_{\gamma_2 \kappa_2}^{*}(x).
\end{equation*}
From \cite{VMK08}
\begin{equation*}
\left( \begin{array}{ccc}
l&l&0 \\ 0&0&0
\end{array} \right) = \frac{1}{\sqrt{2l+1}}
\end{equation*}
and the fact that (see formula \eqref{orth3} or the book by \cite{BalLov09})
\begin{equation*}
\sum_{m=-l}^{+l} (-1)^{m} \left( \begin{array}{ccc} l&l&\gamma_2 \\ m&-m&\kappa_2 \end{array} \right) = \delta_{\gamma_2 \kappa_2}^{00} \sqrt{2l +1}
\end{equation*}
we obtain that
\begin{equation*}
\sum_{m=-l}^{+l} \mathcal{C}^{lm}_{\gamma_2\kappa_2}= \frac{2l+1}{\sqrt{4\pi}} \, \,\delta_{\gamma_2 \kappa_2}^{00}.
\end{equation*}
By recalling that $Y_{00}(x) = 1/\sqrt{4\pi}$  we get that
\begin{equation}
\sum_{m} \mathbb{E} \left[ Y_{lm}(\mathfrak{B}^{x}(t_1))\,Y_{lm}^{*}(\mathfrak{B}^{x}(t_2)) \right] =  \frac{2l+1}{4\pi} \, \exp(-\mu_{l} (t_2-t_1)) \label{ProvCorA}
\end{equation}
which does not depend on $x$ and $t_0$. From this and formula \eqref{ProvCor} we arrive at the claimed result.
\end{proof}

\begin{os}\label{covariance-time-changed-BM}
The (equilibrium) covariance function of the TRD is given by
\begin{equation}
\textsc{Cov}_{\mathfrak{B}}(h; \nu):=\frac{4\pi}{3}\mathbb{E}[Y_{10}\left(\mathfrak{B}^{x}_{\nu}(t_0+h)\right)\, Y_{10}\left(\mathfrak{B}^x_{\nu}(t_0)\right) = \frac{1}{3} E_\nu(- \mu_1\, h^\nu), \quad h >0. \label{eqCov}
\end{equation}

Indeed, from the fact that
\begin{align*}
\P\{ \mathfrak{B}^{x}_{\nu}(t) \in dy\, , \mathfrak{B}^x_{\nu}(t_0) \in dx \} = & \P\{ \mathfrak{B}^x_{\nu}(t_0) \in dx \}\,\P\{ \mathfrak{B}^{x}_{\nu}(t) \in dy\, | \mathfrak{B}^x_{\nu}(t_0) \in dx \} \\
= & \frac{1}{\lambda(\mathbb{S}^2_1)} \P\{ \mathfrak{B}^{x}_{\nu}(t) \in dy\, | \mathfrak{B}^x_{\nu}(t_0) \in dx \},
\end{align*}
we have to consider the integral
\begin{align*}
\frac{1}{\lambda(\mathbb{S}^2_1)} \int_{\mathbb{S}^2_1} \int_{\mathbb{S}^2_1} \cos \vartheta\, \sin \vartheta\, \cos \vartheta_0 \, \sin \vartheta_0\,u_\nu(\vartheta, \varphi, t; \vartheta_0, \varphi_0, t_0) \, d\vartheta\, d\vartheta_0 .
\end{align*}
By taking into account the formula (see \cite{VMK08})
\begin{equation}
Y_{10}(\vartheta, \varphi) = \sqrt{\frac{3}{4\pi}} \cos \vartheta \label{Y10cos}
\end{equation}
and Proposition \ref{propEYY} for $l=1$ and $m=0$, the above integral becomes the covariance \eqref{eqCov}, that is
\begin{equation*}
\textsc{Cov}_{\mathfrak{B}}(h; \nu)= \frac{4\pi}{3}\mathbb{E}[Y_{10}\left(\mathfrak{B}^{x}_{\nu}(t_0+h)\right)\, Y_{10}\left(\mathfrak{B}^x_{\nu}(t_0)\right).
\end{equation*}
Also, for $\nu=1$ we have that
\begin{equation*}
\textsc{Cov}_{\mathfrak{B}}(h; 1)= \frac{1}{3} \exp \left( - \mu_1\, h \right)
\end{equation*}
which means that the rotational Brownian motion has an exponential memory kernel whereas, for $\nu \in (0,1)$, formula \eqref{eqCov} says that the TRD has a covariance function with polynomial memory kernel, in particular $\textsc{Cov}_{\mathfrak{B}}(h; \nu) \sim h^{-\nu}$ as $h \to \infty$. Thus, we obtain that
\begin{equation}
\sum_{h=1}^\infty  \textsc{Cov}_{\mathfrak{B}}(h; \nu) = \infty
\end{equation}
and therefore the TRD has long range dependence for $\nu \in (0,1)$ whereas,
\begin{equation}
\sum_{h=1}^\infty  \textsc{Cov}_{\mathfrak{B}}(h; 1) = \frac{1}{3}\frac{1}{e^{\mu_1} - 1}
\end{equation}
implies  short range dependence for the rotational Brownian motion ($\nu=1$).
\end{os}

\appendix

\section{Background on Clebsch-Gordan coefficients}
\label{AppA}
The Clebsch-Gordan (CG) coefficients are used in mathematics and in particular in representation theory of compact Lie groups. In physics, CG coefficients represent a set of numbers arising in angular momentum coupling under the laws of quantum mechanics. In this work we deal with CG and therefore Wigner coefficients because of their useful connection with integrals of spherical harmonics. For a deep discussion on spherical harmonics and Wigner coefficients we refer to the book \cite{VMK08} while, for the interesting technique related to the graphical theory of angular momentum we refer to the book  \cite{ BalLov09}. Here we only recall some relation which turns out to be useful in the text. We write the CG coefficients by means of the Wigner $3j$-symbols as follows
\begin{equation}
C^{l_3m_3}_{l_1m_1l_2m_2} = (-1)^{l_1-l_2+m_3} \sqrt{2l_3+1}\left( \begin{array}{ccc} l_1&l_2&l_3\\m_1&m_2&-m_3 \end{array} \right) \label{Wig3jSymb}
\end{equation}
with $|m_i | \leq l_i$, $i =1,2,3$. The Wigner's coefficient can be explicitly written as
\begin{align*}
\left( \begin{array}{ccc} l_1&l_2&l_3\\m_1&m_2&m_3 \end{array} \right) = & (-1)^{l_3+m_3+l_2+m_2} \left[ \frac{(l_1+l_2-l_3)!(l_1-l_2+l_3)!(-l_1+l_2+l_3)!}{(l_1+l_2+l_3+1)!} \right]^{1/2}\\
& \times \left[ \frac{(l_3+m_3)!(l_3-m_3)!}{(l_1+m_1)!(l_1-m_1)!(l_2+m_2)!(l_2-m_2)!} \right]^{1/2}\\
& \times \sum_z \frac{(-1)^z (l_2+l_3+m_1-z)!(l_1-m_1+z)!}{z!(l_2+l_3-l_1-z)!(l_3+m_3-z)!(l_1-l_2-m_3+z)!}
\end{align*}
where the summation runs over all $z$'s such that the factorials are non-negative. For the $3j$-symbols we have the following properties.\\
Triangle conditions: the Wigner's 3j coefficients are real-valued and they are different from zero only if $m_1+m_2+m_3=0$.\\
Parity:
\begin{equation}
\left( \begin{array}{ccc} l_1&l_2&l_3\\m_1&m_2&m_3 \end{array} \right) = (-1)^{l_1+l_2+l_3} \left( \begin{array}{ccc} l_1&l_2&l_3\\-m_1&-m_2&-m_3 \end{array} \right).
\end{equation}
Symmetry: for all $l_1,l_2l_3 \geq 0$
\begin{equation}
\left( \begin{array}{ccc} l_1&l_2&l_3\\m_1&m_2&m_3 \end{array} \right) = \left( \begin{array}{ccc} l_2&l_3&l_1\\m_2&m_3&m_1 \end{array} \right) = \left( \begin{array}{ccc} l_3&l_1&l_2\\m_3&m_1&m_2 \end{array} \right).
\end{equation}
Orthonormality:
\begin{equation}
\sum_{m_1m_2} \left( \begin{array}{ccc} l_1 & l_2 & L\\m_1 & m_2 & M \end{array} \right) \left( \begin{array}{ccc}l_1 & l_2 & \xi_1 \\ m_1 & m_2 & \mu_1 \end{array} \right) = \frac{\delta_{L}^{\xi_1}\, \delta_{M}^{\mu_1}}{(2L+1)}, \label{orth1}
\end{equation}

\begin{equation}
\sum_{m} (-1)^{l-m}\left(
\begin{array}{ccc}
l&l&\gamma \\ m&-m&\kappa
\end{array} \right) = \delta_\gamma^0\, \delta_\kappa^0\, \sqrt{2l+1},
\label{orth2}
\end{equation}

\begin{equation}
\sum_{lm} (2l+1) \left( \begin{array}{ccc} l_1&l_2&l\\m_1&m_2&m \end{array} \right) \left( \begin{array}{ccc} l_1&l_2&l\\M_1&M_2&m \end{array} \right) = \delta_{M_1}^{m_1}\, \delta_{M_2}^{m_2}
\label{orth3}
\end{equation}
and
\begin{equation}
\sum_{m_1m_2m_3} \left( \begin{array}{ccc} l_1&l_2&l_3\\m_1&m_2&m_3 \end{array} \right) \left( \begin{array}{ccc} l_1&l_2&l_3\\m_1&m_2&m_3 \end{array} \right) = 1
\label{orth4}
\end{equation}
where
\begin{equation}
\delta_a^b=\left\lbrace \begin{array}{ll} 1, & a=b\\0, & a \neq b \end{array} \right . \label{kronSy}
\end{equation}
is the Kronecker's delta symbol.\\

For the integrals involving spherical harmonics we recall that
\begin{align}
&\int_{0}^{2\pi} d\varphi \int_{0}^{\pi} d\vartheta \, \sin \vartheta\, Y_{l_1m_1}(\vartheta, \varphi)\, Y_{l_2m_2}(\vartheta, \varphi) =  (-1)^{m_2} \delta_{l_1}^{l_2}\delta_{m_1}^{-m_2} \label{appendixInt2Y}
\end{align}
and
\begin{align}
&\int_{0}^{2\pi} d\varphi \int_{0}^{\pi} d\vartheta \, \sin \vartheta\, Y_{l_1m_1}(\vartheta, \varphi)\, Y_{l_2m_2}(\vartheta, \varphi)\, Y_{l_3m_3}(\vartheta, \varphi) \label{intprodtreY} \\
= &\sqrt{\frac{(2l_1+1)(2l_2+1)(2l_3+1)}{4\pi}}  \left( \begin{array}{ccc}l_1&l_2&l_3\\0&0&0 \end{array} \right)  \left( \begin{array}{ccc}l_1&l_2&l_3\\m_1&m_2&m_3 \end{array} \right) \nonumber .
\end{align}



\begin{thebibliography}{99}
\bibitem{AdTay07}
 Adler, R.J., Taylor, J.E.:
\newblock {Random Fields and Geometry}.
\newblock Springer, New York (2007)

\bibitem
{AlloubaNane12}
 Allouba, H.,  Nane, E.:
\newblock {Interacting time-fractional and $\triangle^\nu$ PDEs systems via Brownian-time and inverse-stable-L\'{e}vy-time Brownian sheets}.
\newblock Stoch. Dyn.  (2013). DOI: 10.1142/S0219493712500128

\bibitem
{BM01}
Baeumer, B.,  Meerschaert, M.M.:
\newblock {Stochastic solutions for fractional Cauchy problems}.
\newblock {Fract. Calc. Appl. Anal.}  4, 481 -- 500 (2001)


\bibitem
{BalLov09}
Balcar, E.,  Lovesey, S.W.:
\newblock {Introduction to the Graphical theory of Angular Momentum}.
\newblock Springer, New York (2009)

\bibitem
{BKMP09}
Baldi, P., Kerkyacharian, G., Marinucci, D.,  Picard, D.:
\newblock Asymptotics for spherical needlets.
\newblock {Ann. Statist.} 37,  1150--1171 (2009)

\bibitem
{balMar07}
Baldi, P.,   Marinucci, D.:
\newblock Some characterization of the spherical harmonics coefficients for isotropic random fields.
\newblock {Statist. Probab. Lett.} 77, 490 -- 496 (2007).

\bibitem
{Btoi96}
Bertoin, J.:
\newblock {{L\'{e}vy Processes}}.
\newblock Cambridge University Press, Cambridge (1996).



\bibitem
{CamOrs12}
Cammarota, V.,  Orsingher, E.:
\newblock {Hitting spheres on hyperbolic spaces}. http://arxiv.org/abs/1104.1043 (2011).
Accessed April 2011

 \bibitem
 {Caputo}
   Caputo, M.:  Linear models of dissipation whose Q is almost frequency independent, Part II.
    { Geophys. J. R. Astr. Soc.} { 13}, 529-539 (1967)


\bibitem
{Dodelson}
Dodelson, S.:
\newblock {Modern Cosmology}.
\newblock Academic Press, Boston (2003)

\bibitem
{Dov2}
D'Ovidio, M.:
\newblock Explicit solutions to fractional diffusion equations via generalized gamma convolution.
\newblock {Electron. Commun. Probab.} 15
457 -- 474 (2010)


\bibitem
{Dov4}
D'Ovidio, M.:
\newblock On the fractional counterpart of the higher-order equations.
\newblock {Statist. Probab. Lett.}, 81, 1929 -- 1939 (2011)


\bibitem
{FOX61}
Fox, C.:
\newblock {The G and H functions as symmetrical Fourier kernels}.
\newblock {Trans. Amer. Math. Soc.} 98,  395 -- 429 (1961)

%
%

\bibitem
{karlin-taylor} Karlin, S.,   Taylor, H. M.:  A Second Course in Stochastic Processes.
Academic Press, CAlifornia (1981)


\bibitem
{KST06}
 Kilbas, A.A.,   Srivastava, H.M., Trujillo,  J.J.:
\newblock {{Theory and Applications of Fractional Differential Equations (North-Holland Mathematics Studies)}}, vol. 204.
\newblock Elsevier, Amsterdam (2006)

\bibitem
{Koc89}
 Kochubei, A.~N..
\newblock {The Cauchy problem for evolution equations of fractional order}.
\newblock {Differential Equations}. 25, 967 -- 974 (1989)

\bibitem
{Koc90}
 Kochubei, A.~N.:
\newblock Diffusion of fractional order.
(Russian) Differentsial'nye Uravneniya 26 (1990), no. 4, 660--670, 733--734; translation in Differential Equations 26,  485"1¤72 (1990)


\bibitem
{kolturner}
Kolb, E.,  Turner, M.:
The Early Universe. Westwiev Press, (1994)

%

\bibitem
{krageloh}  Kr\"{a}geloh, A.M.:  Two families of functions related to the fractional powers of generators of strongly continuous contraction semigroups. {J. Math. Anal. Appl.} { 283}, 459-467 (2003)


\bibitem
{meerschaert-skorski}  Leonenko, N. N.,  Meerschaert, M.M.,  Skorskii, A.: Fractional Pearson diffusion.  http://www.stt.msu.edu/users/mcubed/LMS.pdf (2011)

\bibitem
{meerschaert-skorski-correlation}Leonenko, N. N.,  Meerschaert, M.M.,  Skorskii, A.: Correlation structure of  Fractional Pearson diffusions.  http://www.stt.msu.edu/users/mcubed/LMS.pdf (2012)

\bibitem
{MLP01}
Mainardi, F.,  Luchko, Y., Pagnini, G.:
\newblock The fundamental solution of the space-time fractional diffusion equation.
\newblock {Fract. Calc. Appl. Anal.} 4, 153 -- 192 (2001)

\bibitem
{marin06}
Marinucci, D.:
\newblock High-resolution asymptotics for the angular bispectrum of spherical random fields.
\newblock {Ann. Statist.} 34, 1 -- 41 (2006)

\bibitem
{marin08}
Marinucci, D.:
\newblock A central limit theorem and higher order results for the angular bispectrum.
\newblock {Probab. Theory Related Fields. } 141, 389 -- 409 (2008)

\bibitem
{marpec08}
Marinucci, D.,  Peccati, G.:
\newblock {High-frequency asymptotics for subordinated stationary fields on an Abelian compact group }.
\newblock {Stochastic Process. Appl.} 118, 585 -- 613 (2008)

\bibitem
{marpec10}
Marinucci, D.,  Peccati, G.:
\newblock {Representations of $SO(3)$ and angular polyspectra}.
\newblock J. Multivariate Anal. 101, 77 -- 100 (2010)

\bibitem
{MarPeccBook}
Marinucci, D.,  Peccati, G.:
\newblock {Random Fields on the Sphere: Representations, Limit Theorems and Cosmological Applications}.
\newblock Cambridge University Press, New York (2011)

\bibitem
{MSheff04}
Meerschaert, M.M.,   Scheffler, H.P.:
\newblock Limit theorems for continuous time random walks with infinite mean waiting times.
\newblock {J. Appl. Probab.} 41,  623 -- 638 (2004)


\bibitem
{BMN09ann}
 Meerschaert, M.M.,  Nane, E., Vellaisamy, P.:
\newblock {Fractional Cauchy problems on bounded domains}.
\newblock {Ann. Probab.} 37, 979 -- 1007 (2009)

\bibitem
{meerschaert-nane-xiao}
 Meerschaert, M.M.,  Nane, E., Xiao, Y.: Fractal dimensions for continuous time random walk limits.  
http://arxiv.org/abs/1102.0444 (2011)

\bibitem
{meerschaert-skorski-book}  Meerschaert, M.M., Skorskii, A.: Stochastic Models for Fractional Calculus. De Gruyter, Boston, (2012)

\bibitem
{meerschaert-straka} Meerschaert, M.M., Straka, P.: Inverse stable subirdinators.  http://www.stt.msu.edu/users/mcubed/hittingTime.pdf (2012)

\bibitem
{NANERW}
Nane, E.:
\newblock {Fractional Cauchy problems on bounded domains: survey of recent results}. In : Baleanu D. et al (eds.)
\newblock Fractional Dynamics and Control, 185"1¤78, Springer, New York, 2012.

\bibitem
{Nig86}
R.R. Nigmatullin.
\newblock The realization of the generalized transfer in a medium with fractal geometry.
\newblock{Phys. Status Solidi B}. 133, 425 -- 430 (1986)

\bibitem
{OB09}
Orsingher, E., Beghin, L.:
\newblock {Fractional diffusion equations and processes with randomly varying time}.
\newblock {Ann. Probab.} 37,  206 -- 249 (2009)


\bibitem
{PenWil65}
 Penzias, A.A.,  Wilson, R.~W.:
\newblock A measurement of excess antenna temperature at 4080 mc/s.
\newblock {Astrophysical Journal}. 142,  419 -- 421 (1965)

\bibitem
{pod99}
Podlubny, I.:
\newblock {An Introduction to Fractional Derivatives, Fractional Differential Equations, 
Some Methods of Their Solution and Some of Their  Applications}.
\newblock Academic Press, New York (1999)



\bibitem
{SKM93}
 Samko, S.G.,  Kilbas, A.~A.,  Marichev, O.~I.:
\newblock \emph{Fractional Integrals and Derivatives: Theory and Applications}.
\newblock Gordon and Breach, Newark, N. J. (1993)

\bibitem
{SWyss89}
 Schneider, W.R.,  Wyss, W.:
\newblock Fractional diffusion and wave equations.
\newblock {J. Math. Phys.} 30, 134 -- 144 (1989)

\bibitem
{VMK08}
 Varshalovich, D.A.,   Moskalev, A.N.,   Khersonskii, V.K.:
\newblock {Quantum theory of angular momentum}.
\newblock World Scientific Publishing Co. Pte. Ltd., Singapore (2008)


\bibitem
{Wyss86}
Wyss, W.:
\newblock The fractional diffusion equations.
\newblock {J. Math. Phys.} 27, 2782 -- 2785 (1986)


\end{thebibliography}
\end{document}